\documentclass[10pt]{amsart}
\usepackage{amsmath,amsthm,amssymb,xypic,array}
\usepackage[T1]{fontenc}
\usepackage{amsfonts}
\usepackage{amsmath}
\usepackage{amssymb}
\usepackage{amsthm}
\usepackage{setspace}
\usepackage[cp850]{inputenc}
\usepackage{hyperref}
\usepackage{graphicx}
\usepackage{fancyhdr}
\usepackage{tikz}
\usepackage{booktabs}
\usepackage{longtable}
\usepackage{geometry}
\usetikzlibrary{trees}
\providecommand{\U}[1]{\protect\rule{.1in}{.1in}}
\providecommand{\U}[1]{\protect\rule{.1in}{.1in}}
\providecommand{\U}[1]{\protect\rule{.1in}{.1in}}
\providecommand{\U}[1]{\protect\rule{.1in}{.1in}}
\providecommand{\U}[1]{\protect\rule{.1in}{.1in}}
\input{xy}
\xyoption{all}
\def\leq{\leqslant}
\def\geq{\geqslant}

\def\bibaut#1{{\sc #1}}

\def\phi{\varphi}
\def\ro[#1]{{\textcolor{red}{#1}}}

\theoremstyle{theorem}

\newtheorem{Theorem}{Theorem}[section]
\newtheorem{theoremn}{Theorem}

\newtheorem{Lemma}[Theorem]{Lemma}
\newtheorem{Proposition}[Theorem]{Proposition}

\theoremstyle{definition}

\newtheorem{claim}{Claim}

\newtheorem{Construction}[Theorem]{Construction}

\newtheorem{Definition}[Theorem]{Definition}
\newtheorem{Remark}[Theorem]{Remark}

\newtheorem{Example}[Theorem]{Example}

\numberwithin{equation}{section}

\newcommand{\arXiv}[1]{\href{http://arxiv.org/abs/#1}{arXiv:#1}}

\newcommand{\Aut}{\operatorname{Aut}}

\newcommand{\mult}{\operatorname{mult}}

\newcommand{\Exc}{\operatorname{Exc}}

\DeclareMathOperator{\Bl}{Bs}

\DeclareMathOperator{\Ker}{Ker}
\DeclareMathOperator{\im}{Im}

\newcommand\Span[1]{\langle{#1}\rangle}

\newcommand\iso{\cong}

\newcommand{\f}{\varphi}

\renewcommand{\H}{\mathcal{H}}

\renewcommand{\O}{\mathcal{O}}
\renewcommand{\P}{\mathbb{P}}

\begin{document}
\title{On the automorphisms of Hassett's moduli spaces}

\author[Alex Massarenti]{Alex Massarenti}
\address{\sc Alex Massarenti\\
IMPA\\
Estrada Dona Castorina 110\\
22460-320 Rio de Janeiro\\ Brazil}
\email{massaren@impa.br}

\author[Massimiliano Mella]{Massimiliano Mella}
\address{\sc Massimiliano Mella\\ Dipartimento di Matematica e Informatica\\
Universit\`a di Ferrara\\
Via Machiavelli 35\\
44100 Ferrara\\ Italy}
\email{mll@unife.it}

\date{\today}
\thanks{Partially supported by Progetto PRIN 2010 "\textit{Geometria sulle variet\`a algebriche}" MIUR and GRIFGA. This work was done while the first author was a Post-Doctorate at IMPA, funded by CAPES-Brazil.}

\maketitle

\begin{abstract}
Let $\overline{\mathcal{M}}_{g,A[n]}$ be the moduli stack parametrizing weighted stable curves, and let $\overline{M}_{g,A[n]}$ be its coarse moduli space. These spaces have been introduced by \textit{B. Hassett}, as compactifications of $\mathcal{M}_{g,n}$ and $M_{g,n}$ respectively, by assigning rational weights $A = (a_{1},...,a_{n})$, $0< a_{i} \leq 1$ to the markings. In particular, the classical Deligne-Mumford compactification arises for $a_1 = ... = a_n = 1$. In genus zero some of these spaces appear as intermediate steps of the blow-up construction of $\overline{M}_{0,n}$ developed by \textit{M. Kapranov}, while in higher genus they may be related to the LMMP on $\overline{M}_{g,n}$. We compute the automorphism groups of most of the Hassett's spaces appearing in the Kapranov's blow-up construction. Furthermore, if $g\geq 1$ we compute the automorphism groups of all Hassett's spaces. In particular, we prove that if $g\geq 1$ and $2g-2+n\geq 3$ then the automorphism groups of both $\overline{\mathcal{M}}_{g,A[n]}$ and $\overline{M}_{g,A[n]}$ are isomorphic to a subgroup of $S_{n}$ whose elements are permutations preserving the weight data in a suitable sense.
\end{abstract}

\tableofcontents

\section*{Introduction}  
The stack $\overline{\mathcal{M}}_{g,n}$, parametrizing Deligne-Mumford $n$-pointed genus $g$ stable curves, and its coarse moduli space $\overline{M}_{g,n}$ have been among the most studied objects in algebraic geometry for several decades.\\
\indent In \cite{Ha} \textit{B. Hassett} introduced new compactifications $\overline{\mathcal{M}}_{g,A[n]}$ of the moduli stack $\mathcal{M}_{g,n}$ and $\overline{M}_{g,A[n]}$ for the coarse moduli space $M_{g,n}$, by assigning rational weights $A = (a_{1},...,a_{n})$, $0< a_{i}\leq 1$ to the markings. In genus zero some of these spaces appear as intermediate steps of the blow-up construction of $\overline{M}_{0,n}$ developed by \textit{M. Kapranov} in \cite{Ka} and some of them turn out to be Mori Dream Spaces \cite[Section 6]{AM}, while in higher genus they may be related to the LMMP on $\overline{M}_{g,n}$ \cite{Moo}.

In this paper we deal with fibrations and automorphisms of these Hassett's spaces, see also \cite{MM} for the moduli problem that allows some of the sections to have weight zero. Our approach consists in extending some techniques introduced by \textit{A. Bruno} and the authors in \cite{BM1}, \cite{BM2} and \cite{Ma} to study fiber type morphisms from Hassett's spaces, and then apply this knowledge to compute their automorphism groups.

The biregular automorphisms of the moduli space $M_{g,n}$ of $n$-pointed genus $g$-stable curves, and of its Deligne-Mumford compactification $\overline{M}_{g,n}$ have been studied in a series of papers. For instance \cite{BM2}, \cite{Ro}, \cite{Li1}, \cite{Li2}, \cite{GKM} and \cite{Ma}. In \cite{BM1} and \cite{BM2}, \textit{A. Bruno} and the second author, thanks to Kapranov's work \cite{Ka}, managed to study the regular fibrations and the automorphisms of $\overline{M}_{0,n}$ using techniques of classical projective geometry.

Here we push forward these techniques to study Hassett's spaces. The
main novelties are that not all forgetful maps are well defined as
morphisms, and not all permutations of the markings define an
automorphism of the space $\overline{M}_{g,A[n]}$. Indeed in order to
define an automorphism, permutations have to preserve the weight data
in a suitable sense, see Definition \ref{atrans}. We denote by
$\mathcal{S}_{A[n]}$ the group of permutations inducing automorphisms
of $\overline{M}_{g,A[n]}$ and $\overline{\mathcal{M}}_{g,A[n]}$.

As for the ancestors of this paper, \cite{BM2}, \cite{Ma}, the main tool
to study the automorphism groups is a factorization theorem that yields
a morphism of groups between $\Aut(\overline{M}_{g,A[n]})$ and a permutation group. We are able to extend \cite[Theorem 4.1]{BM2} and
\cite[Theorem
0.9]{GKM} to some Hassett's moduli spaces at the expense of  producing a factorization  via
rational maps instead of morphisms. Nonetheless, this is enough to conclude in many cases.

In genus zero we are able 
 to compute the automorphisms of most of the intermediate steps of Kapranov's construction, see Construction \ref{kblu} for the details. 
The results can be summarized as follows.
\begin{theoremn}\label{th:aut0}
For the Hassett's spaces appearing in Construction \ref{kblu} we have
\begin{itemize}
\item[-] $\Aut(\overline{M}_{0,A_{1,n-3}[n]})\cong (\mathbb{C}^{*})^{n-3}\rtimes (S_2\times S_{n-2})$.
\end{itemize}
Furthermore, if $2\leq r\leq n-4$ then:
\begin{itemize}
\item[-] $\Aut(\overline{M}_{0,A_{r,1}[n]})\cong S_{n-r}\times S_r$, 
\item[-] $\Aut(\overline{M}_{0,A_{r,s}[n]})\cong S_{n-r-1}\times S_r$, if $1 < s < n-r-2$,
\item[-] $\Aut(\overline{M}_{0,A_{r,n-r-2}[n]})\cong S_{n-r-1}\times S_{r+1}$. 
\end{itemize}
Finally, if $r = n-3$ then $s =1$, $\overline{M}_{0,A_{n-3,1}[n]}\cong\overline{M}_{0,n}$, and $\Aut(\overline{M}_{0,A_{n-3,1}[n]})\cong S_n$ for any $n\geq 5$.
\end{theoremn}
Note that the Hassett's space $\overline{M}_{0,A_{1,n-3}[n]}$, is the Losev-Manin's moduli space $\overline{L}_{n-2}$ introduced in \cite{LM}. Furthermore, since $\overline{M}_{0,A_{n-3,1}[n]}\cong \overline{M}_{0,n}$ we recover \cite[Theorem 3]{BM2}.

In higher genus, plugging in  arguments from  \cite[Sections
3,4]{Ma}, 
we prove the following statement.
\begin{theoremn}
Let $\overline{\mathcal{M}}_{g,A[n]}$ be the Hassett's moduli stack parametrizing weighted $n$-pointed genus $g$ stable curves, and let $\overline{M}_{g,A[n]}$ be its coarse moduli space. If $g\geq 1$ and $2g-2+n\geq 3$ then
$$\Aut(\overline{\mathcal{M}}_{g,A[n]})\cong\Aut(\overline{M}_{g,A[n]})\cong \mathcal{S}_{A[n]}.$$
Furthermore 
\begin{itemize}
\item[-] $\Aut(\overline{M}_{1,A[2]})\cong (\mathbb{C}^{*})^{2}$ while $\Aut(\overline{\mathcal{M}}_{1,A[2]})$ is trivial,
\item[-] $\Aut(\overline{M}_{1,A[1]})\cong PGL(2)$ while $\Aut(\overline{\mathcal{M}}_{1,A[1]})\cong\mathbb{C}^{*}$.
\end{itemize}
\end{theoremn}
In this case when $A[n] = (1,...,1)$ we recover the main result of \cite{Ma} on the automorphism groups of $\overline{M}_{g,n}$ and $\overline{\mathcal{M}}_{g,n}$.\\  
\indent The paper is organized as follows. In Section \ref{HK} we recall Hassett's and Kapranov's constructions. Section \ref{FIB} is devoted to the proof of the factorization theorems. Finally, in Section \ref{AUT}, we compute the automorphism groups.
\subsubsection*{Acknowledgments}
We would like to thank \textit{J. Hwang} for pointing out a gap in a first version of the factorization theorems.

\section{Hassett's moduli spaces and Kapranov's realizations of $\overline{M}_{0,n}$}\label{HK}
We work over an algebraically closed field of characteristic zero. Let $S$ be a Noetherian scheme and $g,n$ two non-negative integers. A family of nodal curves of genus $g$ with $n$ marked points over $S$ consists of a flat proper morphism $\pi:C\rightarrow S$ whose geometric fibers are nodal connected curves of arithmetic genus $g$, and sections $s_{1},...,s_{n}$ of $\pi$. A collection of input data $(g,A) := (g, a_{1},...,a_{n})$ consists of an integer $g\geq 0$ and the weight data: an element $(a_{1},...,a_{n})\in\mathbb{Q}^{n}$ such that $0<a_{i}\leq 1$ for $i = 1,...,n$, and
$$2g-2 + \sum_{i = 1}^{n}a_{i} > 0.$$
\begin{Definition}\label{defha}
A family of nodal curves with marked points $\pi:(C,s_{1},...,s_{n})\rightarrow S$ is stable of type $(g,A)$ if
\begin{itemize}
\item[-] the sections $s_{1},...,s_{n}$ lie in the smooth locus of $\pi$, and for any subset $\{s_{i_{1}},..., s_{i_{r}}\}$ with non-empty intersection we have $a_{i_{1}} +...+ a_{i_{r}} \leq 1$,
\item[-] $\omega_{\pi}(\sum_{i=1}^{n}a_{i}s_{i})$ is $\pi$-relatively ample, where $\omega_{\pi}$ is the relative dualizing sheaf.
\end{itemize}
\end{Definition}
\textit{B. Hassett} in \cite[Theorem 2.1]{Ha} proved that given a collection $(g,A)$ of input data, there exists a connected Deligne-Mumford stack $\overline{\mathcal{M}}_{g,A[n]}$, smooth and proper over $\mathbb{Z}$, representing the moduli problem of pointed stable curves of type $(g,A)$. The corresponding coarse moduli scheme $\overline{M}_{g,A[n]}$ is projective over $\mathbb{Z}$.\\ 
Furthermore, by \cite[Theorem 3.8]{Ha} a weighted pointed stable curve admits no infinitesimal automorphisms, and its infinitesimal deformation space is unobstructed of dimension $3g-3+n$. Then $\overline{\mathcal{M}}_{g,A[n]}$ is a smooth Deligne-Mumford stack of dimension $3g-3+n$.   
\begin{Remark}\label{normality}
Since $\overline{\mathcal{M}}_{g,A[n]}$ is smooth as a Deligne-Mumford stack the coarse moduli space $\overline{M}_{g,A[n]}$ has finite quotient singularities, that is $\rm\acute{e}$tale locally it is isomorphic to a quotient of a smooth scheme by a finite group. In particular, $\overline{M}_{g,A[n]}$ is normal. 
\end{Remark}
For fixed $g,n$, consider two collections of weight data $A[n],B[n]$ such that $a_i\geq b_i$ for any $i = 1,...,n$. Then there exists a birational \textit{reduction morphism}
$$\rho_{B[n],A[n]}:\overline{M}_{g,A[n]}\rightarrow\overline{M}_{g,B[n]}$$
associating to a curve $[C,s_1,...,s_n]\in\overline{M}_{g,A[n]}$ the curve $\rho_{B[n],A[n]}([C,s_1,...,s_n])$ obtained by collapsing components of $C$ along which $\omega_C(b_1s_1+...+b_ns_n)$ fails to be ample, where $\omega_C$ denote the dualizing sheaf of $C$.\\
Furthermore, for any $g$ consider a collection of weight data $A[n]=(a_1,...,a_n)$ and a subset $A[r]:=(a_{i_{1}},...,a_{i_{r}})\subset A[n]$ such that $2g-2+a_{i_{1}}+...+a_{i_{r}}>0$. Then there exists a \textit{forgetful morphism} 
$$\pi_{A[n],A[r]}:\overline{M}_{g,A[n]}\rightarrow\overline{M}_{g,A[r]}$$
associating to a curve $[C,s_1,...,s_n]\in\overline{M}_{g,A[n]}$ the curve $\pi_{A[n],A[r]}([C,s_1,...,s_n])$ obtained by collapsing components of $C$ along which $\omega_C(a_{i_{1}}s_{i_{1}}+...+a_{i_{r}}s_{i_{r}})$ fails to be ample. For the details see \cite[Section 4]{Ha}.

\begin{Remark}\label{n<3}
If $n\leq 2$ the reduction morphism $\rho:\overline{M}_{g,n}\rightarrow\overline{M}_{g,A[n]}$ contracts at most rational tails with two marked points, such rational tails do not have moduli. Therefore $\rho$ is an isomorphism and $\overline{M}_{g,A[n]}\cong \overline{M}_{g,n}$, see also \cite[Corollary 4.7]{Ha}.
\end{Remark}

\indent In the following we will be especially interested in the boundary of $\overline{M}_{g,A[n]}$. The boundary of $\overline{M}_{g,A[n]}$, as for $\overline{M}_{g,n}$, has a stratification whose loci, called strata, parametrize curves of a fixed topological type and with a fixed configuration of the marked points.\\
We denote by $\Delta_{irr}$ the locus in $\overline{M}_{g,A[n]}$ parametrizing irreducible nodal curves with $n$ marked points, and by $\Delta_{i,P}$ the locus of curves with a node which divides the curve into a component of genus $i$ containing the points indexed by $P$ and a component of genus $g-i$ containing the remaining points. Note that in $\overline{M}_{g,A[n]}$ may appear boundary divisors parametrizing smooth curves. For instance, as soon as there exist two indices $i,j$ such that $a_i+a_j\leq 1$ we get a boundary divisor whose general point represents a smooth curve where the marked points labeled by $i$ and $j$ collide. 

\subsubsection*{Kapranov's blow-up construction}
We follow \cite{Ka}. Let $(C,x_{1},...,x_{n})$ be a genus zero $n$-pointed stable curve. The dualizing sheaf $\omega_{C}$ of $C$ is invertible, see \cite{Kn}. By \cite[Corollaries 1.10 and 1.11]{Kn} the sheaf $\omega_{C}(x_{1}+...+x_{n})$ is very ample and has $n-1$ independent sections. Then it defines an embedding $\phi:C\rightarrow\mathbb{P}^{n-2}$. In particular, if $C\cong\mathbb{P}^{1}$ then $\deg(\omega_{C}(x_{1}+...+x_{n})) = n-2$, $\omega_{C}(x_{1}+...+x_{n})\cong\phi^{*}\mathcal{O}_{\mathbb{P}^{n-2}}(1)\cong\mathcal{O}_{\mathbb{P}^{1}}(n-2)$, and $\phi(C)$ is a degree $n-2$ rational normal curve in $\mathbb{P}^{n-2}$. By \cite[Lemma 1.4]{Ka} if $(C,x_{1},...,x_{n})$ is stable the points $p_{i} = \phi(x_{i})$ are in linear general position in $\mathbb{P}^{n-2}$.\\
This fact combined with a careful analysis of limits in $\overline{M}_{0,n}$ of $1$-parameter families in $M_{0,n}$ led \textit{M. Kapranov} to prove the following theorem \cite[Theorem 0.1]{Ka}:

\begin{Theorem}\label{kaphilb}
Let $p_{1},...,p_{n}\in\mathbb{P}^{n-2}$ be points in linear general position, and let $V_{0}(p_{1},...,p_{n})$ be the scheme parametrizing rational normal curves through $p_{1},...,p_{n}$. Consider $V_{0}(p_{1},...,p_{n})$ as a subscheme of the Hilbert scheme $\mathcal{H}$ parametrizing subschemes of $\mathbb{P}^{n-2}$. Then
\begin{itemize}
\item[-] $V_{0}(p_{1},...,p_{n})\cong M_{0,n}$.
\item[-] Let $V(p_{1},...,p_{n})$ be the closure of $V_{0}(p_{1},...,p_{n})$ in $\mathcal{H}$. Then $V(p_{1},...,p_{n})\cong\overline{M}_{0,n}$. 
\end{itemize}
\end{Theorem}

Kapranov's construction allows to translate many issues of $\overline{M}_{0,n}$ into statements on linear systems on $\P^{n-3}$. Consider a general line $L_{i}\subset\mathbb{P}^{n-2}$ through $p_{i}$. There is a unique rational normal curve $C_{L_{i}}$ through $p_{1},...,p_{n}$, and with tangent direction $L_{i}$ in $p_{i}$. Let $[C,x_{1},...,x_{n}]\in\overline{M}_{0,n}$ be a stable curve, and let $\Gamma\in V_0(p_{1},...,p_{n})$ be the corresponding curve. Since $p_{i}\in\Gamma$ is a smooth point considering the tangent line $T_{p_{i}}\Gamma$, with some work \cite{Ka}, we get a morphism 
$$
\begin{array}{cccc}
f_i: & \overline{M}_{0,n} & \longrightarrow & \mathbb{P}^{n-3}\\
 & [C,x_{1},...,x_{n}] & \longmapsto & T_{p_{i}}\Gamma
\end{array}
$$
Furthermore, $f_{i}$ is birational and it defines an isomorphism on $M_{0,n}$. The birational maps $f_j\circ f_i^{-1}$
  \[
  \begin{tikzpicture}[xscale=1.5,yscale=-1.2]
    \node (A0_1) at (1, 0) {$\overline{M}_{0,n}$};
    \node (A1_0) at (0, 1) {$\mathbb{P}^{n-3}$};
    \node (A1_2) at (2, 1) {$\mathbb{P}^{n-3}$};
    \path (A1_0) edge [->,dashed]node [auto] {$\scriptstyle{f_{j}\circ f_{i}^{-1}}$} (A1_2);
    \path (A0_1) edge [->]node [auto] {$\scriptstyle{f_{j}}$} (A1_2);
    \path (A0_1) edge [->]node [auto,swap] {$\scriptstyle{f_{i}}$} (A1_0);
  \end{tikzpicture}
  \]
are standard Cremona transformations of $\mathbb{P}^{n-3}$ \cite[Proposition 2.12]{Ka}. For any $i = 1,...,n$ the class $\Psi_{i}$ is the line bundle on $\overline{M}_{0,n}$ whose fiber on $[C,x_{1},...,x_{n}]$ is the tangent line $T_{p_{i}}C$. From the previous description we see that the line bundle $\Psi_{i}$ induces the birational morphism $f_{i}:\overline{M}_{0,n}\rightarrow\mathbb{P}^{n-3}$, that is $\Psi_{i} = f_{i}^{*}\mathcal{O}_{\mathbb{P}^{n-3}}(1)$. In \cite{Ka} Kapranov proved that $\Psi_{i}$ is big and globally generated, and that the birational morphism $f_{i}$ is an iterated blow-up of the projections from $p_{i}$ of the points $p_{1},...,\hat{p_{i}},...p_{n}$ and of all strict transforms of the linear spaces they generate, in order of increasing dimension. 
\begin{Construction}\cite{Ka}\label{kblu}
More precisely, fix $(n-1)$-points $p_{1},...,p_{n-1}\in\mathbb{P}^{n-3}$ in linear general position.
\begin{itemize}
\item[(1)] Blow-up the points $p_{1},...,p_{n-2}$, the strict transforms of the lines spanned by two of these $n-2$ points,..., the strict transforms of the linear spaces spanned by the subsets of cardinality $n-4$ of $\{p_{1},...,p_{n-2}\}$.  
\item[(2)] Blow-up $p_{n-1}$, the strict transforms of the lines spanned by pairs of points including $p_{n-1}$ but not $p_{n-2}$,..., the strict transforms of the linear spaces spanned by the subsets of cardinality $(n-4)$ of $\{p_{1},...,p_{n-1}\}$ containing $p_{n-1}$ but not $p_{n-2}$.\\
\vdots
\item[($r$)] Blow-up the strict transforms of the linear spaces spanned by subsets of the form 
$$\{p_{n-1},p_{n-2},...,p_{n-r+1}\}$$ 
so that the order of the blow-ups in compatible by the partial order on the subsets given by inclusion.\\ 
\vdots
\item[($n-3$)] Blow-up the strict transforms of the codimension two linear space spanned by the subset $\{p_{n-1},p_{n-2},...,p_{4}\}$.
\end{itemize}
The composition of these blow-ups is the morphism $f_{n}:\overline{M}_{0,n}\rightarrow\mathbb{P}^{n-3}$ induced by the psi-class $\Psi_{n}$. Identifying $\overline{M}_{0,n}$ with $V(p_{1},...,p_{n})$, and fixing a general $(n-3)$-plane $H\subset\mathbb{P}^{n-2}$, the morphism $f_{n}$ associates to a curve $C\in V(p_{1},...,p_{n})$ the point $T_{p_{n}}C\cap H$.\\
\indent We denote by $W_{r,s}[n]$, where $s = 1,...,n-r-2$, the variety obtained at the $r$-th step once we finish blowing-up the subspaces spanned by subsets $S$ with $|S|\leq s+r-2$, and by $W_{r}[n]$ the variety produced at the $r$-th step. In particular, $W_{1,1}[n] = \mathbb{P}^{n-3}$ and $W_{n-3}[n] = \overline{M}_{0,n}$.
\end{Construction}
In \cite[Section 6.1]{Ha}, Hassett interprets the intermediate steps of Construction \ref{kblu} as moduli spaces of weighted rational curves. Consider the weight data 
$$A_{r,s}[n]:= (\underbrace{1/(n-r-1),...,1/(n-r-1)}_{(n-r-1)-\rm{times}}, s/(n-r-1), \underbrace{1,...,1}_{r-\rm{times}})$$ 
for $r = 1,...,n-3$ and $s = 1,...,n-r-2$. Then $W_{r,s}[n]\cong\overline{M}_{0,A_{r,s}[n]}$, and the Kapranov's map $f_{n}:\overline{M}_{0,n}\rightarrow\mathbb{P}^{n-3}$ factorizes as a composition of reduction morphisms
$$
\begin{array}{l}
\rho_{A_{r,s-1}[n],A_{r,s}[n]}:\overline{M}_{0,A_{r,s}[n]}\rightarrow\overline{M}_{0,A_{r,s-1}[n]},\: s = 2,...,n-r-2,\\ 
\rho_{A_{r,n-r-2}[n],A_{r+1,1}[n]}:\overline{M}_{0,A_{r+1,1}[n]}\rightarrow\overline{M}_{0,A_{r,n-r-2}[n]}.
\end{array} 
$$
\begin{Remark}\label{LM}
The Hassett's space $\overline{M}_{0,A_{1,n-3}[n]}$, that is $\mathbb{P}^{n-3}$ blown-up at all the linear spaces of codimension at least two spanned by subsets of $n-2$ points in linear general position, is the Losev-Manin's moduli space $\overline{L}_{n-2}$ introduced by \textit{A. Losev} and \textit{Y. Manin} in \cite{LM}, see \cite[Section 6.4]{Ha}. The space $\overline{L}_{n-2}$ parametrizes $(n-2)$-pointed chains of projective lines $(C,x_{0},x_{\infty},x_{1},...,x_{n-2})$ where:
\begin{itemize}
\item[-] $C$ is a chain of smooth rational curves with two fixed points $x_{0},x_{\infty}$ on the extremal components,
\item[-] $x_{1},...,x_{n-2}$ are smooth marked points different from $x_{0},x_{\infty}$ but non necessarily distinct,
\item[-] there is at least one marked point on each component.
\end{itemize}
By \cite[Theorem 2.2]{LM} there exists a smooth, separated, irreducible, proper scheme representing this moduli problem. Note that after the choice of two marked points in $\overline{M}_{0,n}$ playing the role of $x_{0},x_{\infty}$ we get a birational morphism $\overline{M}_{0,n}\rightarrow\overline{L}_{n-2}$ which is nothing but a reduction morphism.\\
For example, $\overline{L}_{1}$ is a point parametrizing a $\mathbb{P}^{1}$ with two fixed points and a free point, $\overline{L}_{2}\cong\mathbb{P}^{1}$, and $\overline{L}_{3}$ is $\mathbb{P}^{2}$ blown-up at three points in general position, that is a del Pezzo surface of degree six, see \cite[Section 6.4]{Ha} for further generalizations.
\end{Remark}

\section{Fibrations of $\overline{M}_{g,A[n]}$}\label{FIB}
This section is devoted to study fiber type morphisms of Hassett's moduli spaces. The results are based on Bruno-Mella type argument \cite{BM2} for genus zero, and \cite[Theorem 0.9]{GKM} on fibrations of $\overline{M}_{g,n}$. We start with the genus zero case. For this first we need to understand better $\overline{M}_{0,A[4]}$. 

\begin{Construction}\label{hlw} Fix a weight data $A[4]=(a_1,a_2,a_3,a_4)$ such that 
$a_1\leq a_2\leq a_3\leq a_4$, $a_1+a_2+a_3 >1$ and
$a_1+a_2+a_3+a_4\leq 2$. Under this conditions Hassett's construction can
be extended,  see also \cite[Corollary 4.7]{Ha}, and it is possible to define the  space
$\overline{M}_{0,A[4]}$.
Such a space is isomorphic to
$\overline{M}_{0,4}\cong\mathbb{P}^1$ but the boundary points have a
slightly different interpretation. Indeed the boundary point of
$\overline{M}_{0,4}$ parametrizing  curve with two components and two
marked points on each component corresponds in $\overline{M}_{0,A[4]}$
to a $\mathbb{P}^1$ with two points collapsed and other two marked
points $x_i,x_j$ which can be collapsed or not if either
$a_i+a_j\leq 1$ or $a_i+a_j >1$ respectively. The hypothesis
$a_1\leq a_2\leq a_3\leq a_4$ and $a_1+a_2+a_3 >1$ ensures that three marked points
can not collide. Finally note that as for $\overline{M}_{0,4}$ we have
exactly $\frac{1}{2}\binom{4}{2} = 3$ possible degenerations.\\ 
Fix an admissible weight data $A[n]=(a_1,\ldots,a_4,a_5,\ldots,a_n)$. Then,
under the above assumption, the reduction map
$\rho:\overline{M}_{0,[1,1,1,1,a_5,\ldots,a_n]}\to\overline{M}_{0,A[n]}$ and the forgetful
map $$\pi:\overline{M}_{0,[1,1,1,1,a_5,\ldots,a_n]}\to\overline{M}_{0,4}\iso\P^1$$
allow to define the forgetful map
$\pi_{A[n],A[4]}:\overline{M}_{0,A[n]}\to\overline{M}_{0,A[4]}\iso\P^1$. To see this let
$D_1, D_2$ be two fixed general divisors in the linear system that defines
$\pi$. Then any curve parametrized by $D_i$ can not have two pairs of markings in
$\{1,2,3,4\}$ on two distinct irreducible components otherwise it will
be mapped by $\pi$ in a boundary point of
$\overline{M}_{0,4}$. Therefore $\rho(D_i)$ does not contain the
center of any divisor contracted by $\rho$. This yields
$\rho^{-1}(\rho(D_i))=D_i$ as sets and  $\rho_*(D_1)\cap
\rho_*(D_2)=\emptyset$. In particular, the linear system $|\rho_*(D_1)|$
defines the forgetful morphism $\pi_{A[n],A[4]}$ onto $\P^1\cong\overline{M}_{0,A[4]}$.
\end{Construction}

\begin{Proposition}\label{bpfp}
Let $\overline{M}_{0,A[n]}$ be a Hassett's space admitting a dominant morphism with connected fibers $f:\overline{M}_{0,A[n]}\rightarrow\overline{M}_{0,4}\cong\mathbb{P}^{1}$. Then $f$ factors through a forgetful map. 
\end{Proposition}
\begin{proof}
Let
$f:\overline{M}_{0,A[n]}\rightarrow\overline{M}_{0,4}\cong\mathbb{P}^{1}$
be a dominant morphism and
$\rho_{1}:\overline{M}_{0,n}\rightarrow\overline{M}_{0,A[n]}$ a
reduction morphism. The composition
$f\circ\rho_{1}:\overline{M}_{0,n}\rightarrow\mathbb{P}^{1}$ is a
dominant morphism with connected fibers. By \cite[Theorem 3.7]{BM2}
$f\circ\rho_{1}$ factors through a forgetful map $\pi$. Without loss of generality we may assume $\pi :=\pi_{5,...,n}$. Then we have the following commutative diagram 
  \[
  \begin{tikzpicture}[xscale=4.0,yscale=-1.9]
    \node (A0_1) at (1, 0) {$\overline{M}_{0,n}$};
    \node (A1_1) at (1, 1) {$\overline{M}_{0,A[n]}$};
    \node (A2_0) at (0, 2) {$\overline{M}_{0,4}$};
    \node (A2_1) at (1, 2) {$\overline{M}_{0,4}\cong\mathbb{P}^{1}$};
    \path (A1_1) edge [->]node [auto] {$\scriptstyle{f}$} (A2_1);
    \path (A2_0) edge [->]node [auto] {$\scriptstyle{\phi}$} (A2_1);
   \path (A1_1) edge [->,swap]node [auto] {$\overline{\scriptstyle{f}}$} (A2_0);
    \path (A0_1) edge [->]node [auto] {$\scriptstyle{\rho_{1}}$} (A1_1);
    \path (A0_1) edge [->]node [auto,swap] {$\scriptstyle{\pi}$} (A2_0);
  \end{tikzpicture}
  \] 
where $\phi\in\Aut(\mathbb{P}^{1})$ and $\overline{f}=\f^{-1}\circ f$.  
Reorder the
weight data in such a way that $a_1\leq a_2\leq a_3\leq a_4$. Let $[D_p]$
be the point in $\overline{M}_{0,n}$ corresponding to a curves
$C_1\cup C_2$ with markings $(1,2,3)$ on $C_1\cong\mathbb{P}^1$, the
remaining on $C_2$, $p=C_1\cap C_2$, and $C_2$ have $n-4$
components. Then when the attaching point $p$ varies in $C_1$, the
points $[D_p]$ spans a curve $\Gamma\subset\overline{M}_{0,n}$ and
$\pi_{|\Gamma}$ is dominant. Since the diagram is commutative $\rho_1$
can not contract $\Gamma$, therefore $a_1+a_2+a_3>1$.\\ 
Therefore, by Construction \ref{hlw} we have a well defined Hassett's
space $\overline{M}_{0,A[4]}$ with a forgetful morphism
$\pi_{A[n],A[4]}:\overline{M}_{0,A[n]}\rightarrow\overline{M}_{0,A[4]}\iso\overline{M}_{0,4}$. Since $\overline{f}$ and $\pi_{A[n],A[4]}$ coincide on the open subset
$M_{0,A[n]}$ they are equal, and $f = \phi\circ \pi_{A[n],A[4]}$. 
\end{proof}
To extend Proposition \ref{bpfp} to a wider contest we use the following result that compares forgetful maps.

\begin{Lemma}\label{lem:forg1}
Fix a general point $x\in \overline{M}_{0,A[n]}$. Let $\pi_{I_1}:\overline{M}_{0,A[n]}\to \overline{M}_{0,A[r]}$ and $\pi_{I_2}:\overline{M}_{0,A[n]}\to \overline{M}_{0,A[r]}$ be two forgetful morphisms, and $F_{i,x}$ the corresponding fibers passing through $x$ for $i=1,2$. If $\dim (F_{1,x}\cap F_{2,x})=n-r-1$ and $r>4$ then $\pi_{I_1}$ and $\pi_{I_2}$ forget a common set of $r-1$ points.
\end{Lemma}
\begin{proof}
Let $I$ be the, possibly empty, set of common indexes. Then we may factor $\pi_{I_i}$ via  $\pi_I:\overline{M}_{0,A[n]}\rightarrow \overline{M}_{0,A[m]}$ and $\pi_{i}^\prime:\overline{M}_{0,A[m]}\rightarrow\overline{M}_{0,A[r]}$. In particular the morphisms $\pi_{1}^\prime$ and $\pi_2^\prime$ forget disjoint set of indexes.\\
Let $y=\pi_I(x)$ and  $F_{i,y}^\prime$ be the associated fibers through $y$. Assume that $m>r+1\geq 6$. Then we may factor $\pi_{i}^\prime$ with forgetful maps that forgets two disjoint pairs of indexes. This yields $\dim (F_{1,y}^\prime\cap F_{2,y}^\prime)\leq r-2$. Therefore $m=r+1$ and we conclude. 
\end{proof}
The following is the result we were looking for in the genus zero case.
\begin{Theorem}\label{facg0}Let $f:\overline{M}_{0,A[n]}\rightarrow\overline{M}_{0,B[r]}$ be a dominant morphism with connected fibers. Assume that either $r=4$ or $r>4$ and the weight data $(b_1,\ldots, b_r)$ satisfies , after reordering the indexes, $b_1\leq b_2\leq\ldots\leq b_r$, and  $b_r+b_{r-1}+b_{r-2}+b_{r-4}>2$.
Then $f$ factors, as a rational map, through a forgetful map $\pi_I:\overline{M}_{0,A[n]}\rightarrow\overline{M}_{0,A[r]}$ and a birational map $\overline{\phi}:\overline{M}_{0,A[r]}\dasharrow\overline{M}_{0,B[r]}$.
\end{Theorem}
\begin{proof}
We proceed by induction on $\dim\overline{M}_{0,B[r]}$. The first step of the induction, that is $\dim\overline{M}_{0,B[r]} = 1$, is Proposition \ref{bpfp}.\\ 
By hypothesis the forgetful  maps $\pi_{i}:\overline{M}_{0,B[r]}\rightarrow\overline{M}_{0,B[r-1]}$ forgetting the marked point $x_i$ with $i=1,2$ are well defined, and the space $\overline{M}_{0,B[r-1]}$ satisfies the induction hypothesis. The morphism $\pi_{i}\circ f$ is dominant, with connected fibers, hence by induction hypothesis it factors, as a rational map, through a forgetful morphism $\pi_{I_i}:\overline{M}_{0,A[n]}\rightarrow\overline{M}_{0,A[r-1]}$. Let $x\in \overline{M}_{0,A[n]}$ be a general point and $F_{i,x}$ the fiber of $\pi_{I_i}$ through $x$. By construction we have $F_{1,x}\cap F_{2,x}\supset f_x$, where $f_x$ is the fiber of $f$ through $x$. Therefore, by Lemma \ref{lem:forg1} the subsets $I_i$ share $n-r$ common indexes, let $I_0$ be the set of these indexes. By construction $\dim F_{\pi_{I_0}}=\dim F_f$ this shows that the general fiber of
$\pi_{I_0}$ is contracted by $f$ and the general fiber of $f$ is contracted by $\pi_{I_0}$. 
\end{proof}
Next we concentrate on higher genera. If $g\geq 1$ then all forgetful morphisms are always well defined. Therefore, the following is just a simple adaptation of \cite[Theorem 0.9]{GKM}.
\begin{Proposition}\label{gkmha}
Let $f:\overline{M}_{g,A[n]}\rightarrow X$ be a dominant morphism with connected fibers. If $g\geq 1$ either $f$ is of fiber type and factors, as a rational map, through a forgetful morphism $\pi_{I}:\overline{M}_{g,A[n]}\rightarrow\overline{M}_{g,A[r]}$ and a birational map $\overline{\phi}:\overline{M}_{g,A[r]}\dasharrow X$, or $f$ is birational and $\Exc(f)\subseteq\partial\overline{M}_{g,A[n]}$.
\end{Proposition}
\begin{proof}
By \cite[Theorem 4.1]{Ha} any Hassett's moduli space $\overline{M}_{g,A[n]}$ receives a birational reduction morphism $\rho_{n}:\overline{M}_{g,n}\rightarrow\overline{M}_{g,A[n]}$ restricting to the identity on $M_{g,n}$. The composition $f\circ\rho_{n}:\overline{M}_{g,n}\rightarrow X$ gives a fibration of $\overline{M}_{g,n}$ to a projective variety.\\ 
If $f$ is of fiber type by \cite[Theorem 0.9]{GKM} the morphism $f\circ\rho_{n}$ factors, as a rational map, through a forgetful map $\pi_{i}:\overline{M}_{g,n}\rightarrow\overline{M}_{g,i}$, with $i < n$, and a morphism $\alpha:\overline{M}_{g,i}\rightarrow X$. Considering the corresponding forgetful map $\pi_{i}^{H}:\overline{M}_{g,A[n]}\rightarrow\overline{M}_{g,A[i]}$ on the Hassett's spaces, and another birational morphism $\rho_{i}:\overline{M}_{g,i}\rightarrow\overline{M}_{g,A[i]}$ restricting to the identity on $M_{g,i}$, we get the following commutative diagram:
  \[
  \begin{tikzpicture}[xscale=2.5,yscale=-1.2]
    \node (A0_0) at (0, 0) {$\overline{M}_{g,n}$};
    \node (A0_1) at (1, 0) {$\overline{M}_{g,i}$};
    \node (A1_0) at (0, 1) {$\overline{M}_{g,A[n]}$};
    \node (A1_1) at (1, 1) {$\overline{M}_{g,A[i]}$};
    \node (A2_2) at (2, 2) {$X$};
    \path (A0_0) edge [->]node [auto] {$\scriptstyle{\pi_{i}}$} (A0_1);
    \path (A1_0) edge [->]node [auto] {$\scriptstyle{\pi_{i}^{H}}$} (A1_1);
    \path (A1_0) edge [->,bend left=20,swap]node [auto] {$\scriptstyle{f}$} (A2_2);
    \path (A1_1) edge [->,dashed]node [auto] {$\scriptstyle{\overline{\phi}}$} (A2_2);
    \path (A0_0) edge [->,swap]node [auto] {$\scriptstyle{\rho_{n}}$} (A1_0);
    \path (A0_1) edge [->]node [auto] {$\scriptstyle{\rho_{i}}$} (A1_1);
    \path (A0_1) edge [->,bend right=20]node [auto] {$\scriptstyle{\alpha}$} (A2_2);
  \end{tikzpicture}
  \]
Note that $\rho_{i}\circ\pi_{i}$ and $\pi_{i}^{H}\circ\rho_{n}$ are defined on $\overline{M}_{g,n}$ and coincide on $M_{g,n}$. Since $\overline{M}_{g,n}$ is separated we have $\rho_{i}\circ\pi_{i} = \pi_{i}^{H}\circ\rho_{n}$. Let $\overline{\phi}:\overline{M}_{g,A[i]}\dasharrow X$ be the birational map induced by $\alpha$. Clearly $\overline{\phi}\circ\pi_{i}^{H} = f$.\\
Now, assume that $f$ is birational. If $\Exc(f) \cap M_{g,A[n]}\neq\emptyset$ then $\Exc(f\circ\rho_{n})\cap M_{g,n}\neq\emptyset$. This contradicts \cite[Theorem 0.9]{GKM}. So $\Exc(f)\subseteq\partial\overline{M}_{g,A[n]}$.\\ 
Let us consider the case $g = 1$. If $f$ is of fiber type, by the second part of \cite[Theorem 0.9]{GKM}, the fibration $f\circ\rho_{n}$ factors, as a rational map, through $\pi_{I}\times\pi_{I^{c}}$. Let $\overline{\phi}:\overline{M}_{1,A[i]}\times_{\overline{M}_{1,A[m]}}\overline{M}_{1,A[n-i]}\dasharrow X$ the birational map induced by $\alpha$. Therefore, we have the following commutative diagram
 \[
  \begin{tikzpicture}[xscale=4.5,yscale=-1.2]
    \node (A0_0) at (0, 0) {$\overline{M}_{1,n}$};
    \node (A0_1) at (1, 0) {$\overline{M}_{1,S}\times_{\overline{M}_{1,m}}\overline{M}_{1,S^{c}}$};
    \node (A1_0) at (0, 1) {$\overline{M}_{1,A[n]}$};
    \node (A1_1) at (1, 1) {$\overline{M}_{1,A[i]}\times_{\overline{M}_{1,A[m]}}\overline{M}_{1,A[n-i]}$};
    \node (A2_2) at (2, 2) {$X$};
    \path (A0_0) edge [->]node [auto] {$\scriptstyle{\pi_{I}\times\pi_{I^{c}}}$} (A0_1);
    \path (A1_0) edge [->]node [auto] {$\scriptstyle{\pi_{I}^{H}\times\pi_{I^{c}}^{H}}$} (A1_1);
    \path (A1_0) edge [->,bend left=20,swap]node [auto] {$\scriptstyle{f}$} (A2_2);
    \path (A1_1) edge [->,dashed]node [auto] {$\scriptstyle{\overline{\phi}}$} (A2_2);
    \path (A0_0) edge [->,swap]node [auto] {$\scriptstyle{\rho_{n}}$} (A1_0);
    \path (A0_1) edge [->]node [auto] {$\scriptstyle{\rho_{i}\times\rho_{i^{c}}}$} (A1_1);
    \path (A0_1) edge [->,bend right=20]node [auto] {$\scriptstyle{\alpha}$} (A2_2);
  \end{tikzpicture}
  \]
We may assume that the forgetful morphisms $\overline{M}_{1,A[i]}\rightarrow\overline{M}_{1,A[m]}$, $\overline{M}_{1,A[n-i]}\rightarrow\overline{M}_{1,A[m]}$ realizing the fiber product forget respectively the first $i-m$ and $(n-i)-m$ marked points. Therefore, we can identify a point in $\overline{M}_{1,A[i]}\times_{\overline{M}_{1,A[m]}}\overline{M}_{1,A[n-i]}$ with a pair of the form: 
$$([C,x_1,...,x_i],[C,x_1,...,x_m,x_{m+1},...,x_i,x_{m+i+1},...,x_{n}]).$$ 
Now, we have a birational map
$$
\begin{array}{cccc}
\tilde{\phi}: & \overline{M}_{1,A[i]}\times_{\overline{M}_{1,A[m]}}\overline{M}_{1,A[n-i]} & \longrightarrow & \overline{M}_{1,A[n-m]}\\
 & ([C,x_1,...,x_i],[C,x_1,...,x_m,x_{m+1},...,x_i,x_{m+i+1},...,x_{n}]) & \longmapsto & [C,x_1,...,x_i,x_{m+i+1},...,x_{n}]
\end{array}
$$
Let $\pi_{I}:\overline{M}_{1,A[n]}\rightarrow\overline{M}_{1,A[n-m]}$ be the forgetful morphism forgetting the marked points indexed by $I = \{i+1,...,i+m\}$. Then $f$ factors, as a rational map, as $\overline{\phi}\circ\tilde{\phi}^{-1}\circ\pi_{I}$.\\
If $f$ is birational and $\Exc(f)\cap M_{1,A[n]}\neq\emptyset$ then $\Exc(f\circ\rho_{n})\cap M_{1,n}\neq\emptyset$. Again this contradicts the second part of \cite[Theorem 0.9]{GKM}. So $\Exc(f)\subseteq\partial\overline{M}_{1,A[n]}$.
\end{proof}

\section{Automorphisms of $\overline{M}_{g,A[n]}$}\label{AUT}
Let $\f:\overline{M}_{g,A[n]}\to\overline{M}_{g,A[n]}$ be an automorphism and $\pi_i:\overline{M}_{g,A[n]}\to\overline{M}_{g,A[n-1]}$ a forgetful morphism. 
Consider the composition $\pi_i\circ\phi^{-1}$ and assume we know that $\pi_i\circ\phi^{-1}$ factors via a forgetful map $\pi_{j_i}$. Then this produces the following commutative diagram
$$
  \begin{tikzpicture}[xscale=3.5,yscale=-1.2]
    \node (A0_0) at (0, 0) {$\overline{M}_{g,A_{[n]}}$};
   \node (A0_1) at (1, 0) {$\overline{M}_{g,A_{[n]}}$};
    \node (A1_0) at (0, 1) {$\overline{M}_{g,A_{[n-1]}}$};
    \node (A1_1) at (1, 1) {$\overline{M}_{g,A_{[n-1]}}$};
    \path (A0_0) edge [->]node [auto] {$\scriptstyle{\phi^{-1}}$} (A0_1); 
\path (A0_1) edge [->]node [auto] {$\scriptstyle{\pi_i}$} (A1_1);
 \path (A0_0) edge [->]node [auto,swap] {$\scriptstyle{\pi_{j_i}}$} (A1_0);
 \path (A1_0) edge [->,dashed]node [auto] {$\scriptstyle{\tilde{\phi}}$} (A1_1);
\end{tikzpicture}
$$
In particular this associates to $\phi$ the transposition $i\leftrightarrow j_i$ and, if we may apply this argument to $m$ forgetful maps, it produces a morphism
\begin{equation}\label{eq:chi}
\chi_{g,A[n]}:\Aut(\overline{M}_{g,A[n]})\rightarrow S_m. 
\end{equation}
This argument, together with Theorem \ref{facg0}, and Proposition \ref{gkmha} gives rise to meaningful morphisms $\chi_{g,A[n]}$  for a vast class of Hassett's spaces. The aim of this section is to determine the automorphism group of many Hassett's spaces studying image and kernel of these morphisms. 

We start considering some Hassett's spaces of Construction \ref{kblu}.\\
Assume $r = 1$ and $s = n-3$. Then the weight data is
$$A_{1,n-3}[n]:= (\underbrace{1/(n-2),...,1/(n-2)}_{(n-2)-\rm{times}}, (n-3)/(n-2), 1).$$ 
Therefore, on $\overline{M}_{0,A_{1,n-3}[n]}$ the forgetful map $\pi_i$ is well defined if and only if $i\leq n-2$. In particular, after any such forgetful map the image satisfies the hypothesis of Theorem \ref{facg0}. This yields the morphism
\begin{equation}\label{mor1}
\chi_{1,n-3}:=\chi_{0,A_{1,n-3}[n]}:\Aut(\overline{M}_{0,A_{1,n-3}[n]})\rightarrow S_{n-2}.
\end{equation}
If $r\geq 2$ then the weight data is 
\begin{equation}\label{wdr2}
A_{r,s}[n]:= (\underbrace{1/(n-r-1),...,1/(n-r-1)}_{(n-r-1)-\rm{times}}, s/(n-r-1), \underbrace{1,...,1}_{r- \rm{times}})
\end{equation}
and on $\overline{M}_{0,A_{r,s}[n]}$ any forgetful map is well defined. If $r = 2$ and $s\leq n-5$ the target space of the first $n-2$ forgetful maps satisfies the hypothesis of Theorem \ref{facg0}. This yields the morphism 
\begin{equation}\label{mor2}
\chi_{2,s}:=\chi_{0,A_{2,s}[n]}:\Aut(\overline{M}_{0,A_{2,s}[n]})\rightarrow S_{n-2}
\end{equation}
If either $r = 2$ and $s = n-4$, or $r\geq 3$ after any forgetful map the target space satisfies the hypothesis of Theorem \ref{facg0}. This yields the morphism 
\begin{equation}\label{mor3}
\chi_{r,s}:=\chi_{0,A_{r,s}[n]}:\Aut(\overline{M}_{0,A_{r,s}[n]})\rightarrow S_{n}
\end{equation}
Now, we consider curves of positive genus. Observe that $\overline{M}_{1,A[1]}\cong\overline{M}_{1,1}\cong\mathbb{P}^{1}$ for any weight data. Therefore we may restrict to the cases $g = 1, n\geq 2$ and $g\geq 2, n\geq 1$.
\begin{Lemma}\label{efm}
If $g = 1, n\geq 2$ or $g\geq 2, n\geq 1$ then all the forgetful morphisms $\overline{M}_{g,A[n]}\rightarrow\overline{M}_{g,A[n-1]}$ are well defined. 
\end{Lemma}
\begin{proof}
If $g = 1$ then $2g-2+a_{1}+...+a_{n-1} = a_{1}+...+a_{n-1}>0$ being $n\geq 2$. If $g = 2$ we have $2g-2+a_{1}+...+a_{n-1}\geq 2 + a_{1}+...+a_{n-1}>0$ for any $n\geq 1$. To conclude it is enough to apply \cite[Theorem 4.3]{Ha}.   
\end{proof}
Then Proposition \ref{gkmha}, and Lemma \ref{efm} yield the morphism
\begin{equation}\label{mor4}
\chi_{g,A[n]}: \Aut(\overline{M}_{g,A[n]})\rightarrow S_{n}
\end{equation}

Next we have to determine the image of the morphisms $\chi_{g,A[n]}$. Let us start with an example.

\begin{Example}
In $\overline{M}_{2,A[4]}$ with weights $(1,1/3,1/3,1/3)$ consider the divisor parametrizing reducible curves $C_{1}\cup C_{2}$, where $C_{1}$ has genus zero and markings $(1,1/3,1/3)$, and $C_{2}$ has genus two and marking $1/3$.
$$\includegraphics[scale=0.4]{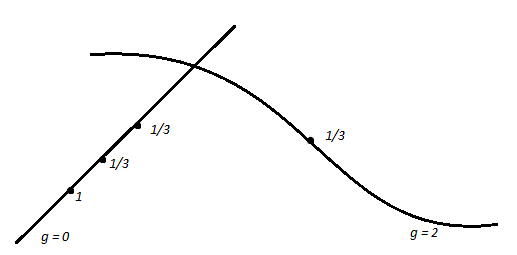}$$
After the transposition $1\leftrightarrow 4$ the genus zero component has markings $(1/3,1/3,1/3)$, so it is contracted. This means that the transposition induces a birational map
 \[
  \begin{tikzpicture}[xscale=2.5,yscale=-1.2]
    \node (A0_0) at (0, 0) {$\overline{M}_{2,A[4]}$};
    \node (A0_1) at (1, 0) {$\overline{M}_{2,A[4]}$};
    \path (A0_0) edge [->,dashed]node [auto] {$\scriptstyle{1\leftrightarrow 4}$} (A0_1);
  \end{tikzpicture}
  \]
contracting a divisor on a codimension two subscheme of $\overline{M}_{2,A[4]}$. Consider the locus of smooth genus two curves $C$ with four marked points such that $x_{2} = x_{3} = x_{4}\in C$. Since $a_{1}+a_{2}+a_{3}>1$ the birational map induced by $1\leftrightarrow 4$ is not defined on such locus.
\end{Example}
This example suggests that troubles come from rational tails with at least three marked points, and leads us to the following definition.
\begin{Definition}\label{atrans}
A transposition $i\leftrightarrow j$ of two marked points is \textit{admissible} if and only if for any $h_{1},...,h_{r}\in\{1,...,n\}$, with $r\geq 2$,
$$a_{i}+\sum_{k = 1}^{r}a_{h_{k}}\leq 1 \iff a_{j}+\sum_{k = 1}^{r}a_{h_{k}}\leq 1.$$
\end{Definition}

\begin{Lemma}\label{lsal}
Let $\sigma\in S_n$ be a permutation inducing an automorphism of $\overline{M}_{g,A[n]}$. Assume that if $i\neq\sigma(i)$ the forgetful morphisms $\pi_i$ and $\pi_{\sigma(i)}$ are well defined. Let  $A_i[n-1] = (a_1,...,\hat{a}_{i},...,n)$ and $A_{\sigma(i)}[n-1] = (a_{\sigma(1)},...,\hat{a}_{\sigma(i)},...,a_{\sigma(n)})$ be weight data, then $\sigma$ induces an isomorphism
 $$\overline{M}_{g,A_i[n-1]}\cong \overline{M}_{g,A_{\sigma(i)}[n-1]},$$ 
whenever $\overline{M}_{g,A_i[n-1]}$ is well defined.
\end{Lemma}
\begin{proof}
If $i=\sigma(i)$ there is nothing to prove. Assume that $i\neq\sigma(i)$, then we have the following commutative diagram
  \[
  \begin{tikzpicture}[xscale=3.5,yscale=-1.2]
    \node (A0_0) at (0, 0) {$\overline{M}_{g,A[n]}$};
    \node (A0_1) at (1, 0) {$\overline{M}_{g,A[n]}$};
    \node (A1_0) at (0, 1) {$\overline{M}_{g,A_i[n-1]}$};
    \node (A1_1) at (1, 1) {$\overline{M}_{g,A_{\sigma(i)}[n-1]}$};
    \path (A0_0) edge [->]node [auto] {$\scriptstyle{\sigma}$} (A0_1);
    \path (A0_0) edge [->,swap]node [auto] {$\scriptstyle{\pi_i}$} (A1_0);
    \path (A0_1) edge [->]node [auto] {$\scriptstyle{\pi_{\sigma(i)}}$} (A1_1);
    \path (A1_0) edge [->,dashed]node [auto] {$\scriptstyle{\tilde{\sigma}}$} (A1_1);
  \end{tikzpicture}
  \]
Where, $\tilde{\sigma}\in S_{n-1}$ is the permutation induced by $\sigma$. Let us assume that the indeterminacy locus of $\tilde{\sigma}$ is not empty. This means that there are $\{j_1,...,j_r\}\subseteq \{1,...,\hat{i},...,n\}$ with $r\geq 3$ such that $a_{j_1}+...+a_{j_r}\leq 1$ but $a_{\sigma(j_1)}+...+a_{\sigma(j_r)} > 1$. In particular, there is a subvariety $Z\subset \overline{M}_{g,A_i[n-1]}$ whose general point corresponds to an irreducible curve $C$ of genus $g$ with $x_{j_1}=...=x_{j_r}$ and other marked points.

Let $w\in\pi_i^{-1}(Z)$ be a general point parametrizing a curve $C$ with marked points $x_{j_1}=...=x_{j_r}$, $x_i$ and eventually other marked points. By hypothesis $a_{\sigma(j_1)}+...+a_{\sigma(j_r)} > 1$, hence we derive the contradiction that  the morphism $\sigma$ is not defined on $w$. By the same argument applied to $\sigma^{-1}$ we conclude that  $\tilde{\sigma}$ is an isomorphism.

\end{proof}

\begin{Remark}\label{autfl}
If $\phi:X\rightarrow X$ is an automorphism of a scheme then the fixed points of $\phi$ form a closed subscheme. So once the fixed locus includes a dense set $U\subseteq X$, the fixed locus is the entire space, and $\phi$ is the identity.
\end{Remark}

\begin{Proposition}\label{k2015}
Let $\overline{M}_{g,A[n]}$ be a Hassett's space such that $2g-2+n\geq 3$. Assume that  in Equation~(\ref{eq:chi}) we have $m = n$. Then $\ker(\chi_{g,A[n]})$ is trivial.\\ Furthermore, if $\sigma_{\phi} = \chi_{g,A[n]}(\phi)$ is a permutation induced by an automorphism $\phi\in \Aut(\overline{M}_{g,A[n]})$, then $\sigma_{\phi}$ induces an automorphism of $\overline{M}_{g,A[n]}$, and $\phi$ acts on $\overline{M}_{g,A[n]}$ as the permutation $\sigma_{\phi}$. 
\end{Proposition}
\begin{proof}
By Remark \ref{n<3} and \cite{Ma} we may assume $n\geq 3$. Let $\phi\in\Aut(\overline{M}_{g,A[n]})$ be an automorphism such that $\chi_{g,A[n]}(\phi)$ is the identity, that is for any $i = 1,...,n$ the fibration $\pi_{i}\circ\phi^{-1}$, and the fibration $\pi_{i}\circ\phi$ as well, factor through $\pi_{i}$ and we have $n$ commutative diagrams
  \[
  \begin{tikzpicture}[xscale=2.7,yscale=-1.2]
    \node (A0_0) at (0, 0) {$\overline{M}_{g,A[n]}$};
    \node (A0_1) at (1, 0) {$\overline{M}_{g,A[n]}$};
    \node (A1_0) at (0, 1) {$\overline{M}_{g,A_1[n-1]}$};
    \node (A1_1) at (1, 1) {$\overline{M}_{g,A_1[n-1]}$};
    \path (A0_0) edge [->]node [auto] {$\scriptstyle{\phi}$} (A0_1);
    \path (A1_0) edge [->,dashed]node [auto] {$\scriptstyle{\tilde{\phi}_{1}}$} (A1_1);
    \path (A0_1) edge [->]node [auto] {$\scriptstyle{\pi_{1}}$} (A1_1);
    \path (A0_0) edge [->,swap]node [auto] {$\scriptstyle{\pi_{1}}$} (A1_0);
  \end{tikzpicture}
  \quad
  \cdots
  \quad
  \begin{tikzpicture}[xscale=2.7,yscale=-1.2]
    \node (A0_0) at (0, 0) {$\overline{M}_{g,A[n]}$};
    \node (A0_1) at (1, 0) {$\overline{M}_{g,A[n]}$};
    \node (A1_0) at (0, 1) {$\overline{M}_{g,A_n[n-1]}$};
    \node (A1_1) at (1, 1) {$\overline{M}_{g,A_n[n-1]}$};
    \path (A0_0) edge [->]node [auto] {$\scriptstyle{\phi}$} (A0_1);
    \path (A1_0) edge [->,dashed]node [auto] {$\scriptstyle{\tilde{\phi}_{n}}$} (A1_1);
    \path (A0_1) edge [->]node [auto] {$\scriptstyle{\pi_{n}}$} (A1_1);
    \path (A0_0) edge [->,swap]node [auto] {$\scriptstyle{\pi_{n}}$} (A1_0);
  \end{tikzpicture}
  \]
Let $[C,x_1,...,x_n]\in \overline{M}_{g,A[n]}$ be a general point, and let $[\Gamma,y_1,...,y_n] = \phi([C,x_1,...,x_n])$ be its image via $\phi$. Then $\tilde{\phi}_i([C,x_1,...,\hat{x}_i,...,x_n]) = [\Gamma,y_1,...,\hat{y}_i,...,y_n]$ for any $i = 1,...,n$. Now let $F_{1,x} = \pi_1^{-1}(\pi_1([C,x_1,...,x_n]))$ and $F_{1,y} = \pi_1^{-1}(\pi_1([\Gamma,y_1,...,y_n]))$. Therefore $\phi$ induces an isomorphism $\phi_{|F_{1,x}}:F_{1,x}\rightarrow F_{1,y}$.\\
Let us consider the fiber $F_{1,x}$. Since $2g-2+n\geq 3$ the general $(n-1)$-pointed genus $g$ curve is automorphism-free. Therefore $F_{1,x}$ is isomorphic to $C$, and on this fiber we have $n-1$ special points, which we still denote by $x_2,...,x_n$, corresponding to curves where $x_1$ and $x_i$ collide for $i = 2,...,n$. Clearly, we have the same picture for $F_{1,y}$ and we denote by $y_2,...,y_n$ the corresponding special points.\\
Now we consider $\phi(x_2)$, and distinguish two cases:
\begin{itemize}
\item[(\textit{a})] Assume that $\phi(x_2)\in F_{1,y}\setminus\{y_2,...,y_n\}$. Then $\phi(x_2) = [\Gamma,z_1,y_2,...,y_n]$ is the class of a smooth curve with $n$ distinct marked points. Then $\pi_3(\phi_{|F_{1,x}}(x_2)) = [\Gamma,z_1,y_2,y_4,...,y_n]$ and since there are no two collapsed points on $\Gamma$ and $(\Gamma,z_1,y_2,y_4,...,y_n)$ is automorphism-free we have $\pi_3^{-1}([\Gamma,z_1,y_2,y_4,...,y_n])\cong \Gamma$.\\
On the other hand $\pi_3\circ\phi$ factors through $\pi_3$. We know that $\pi_3(x_2)$ is the class of a curve $[C^{'},x_1,x_2,x_4,...,x_n]$, where either $x_1$ and $x_2$ are collapsed and $C^{'}\cong C$ is irreducible or $x_1\neq x_2$ and $C^{'}\cong C\cup \mathbb{P}^1$ with $x_1,x_2$ on the rational tail, depending on whether $a_1+a_2\leq 1$ or $a_1+a_2> 1$. We have the following commutative diagram
\[
  \begin{tikzpicture}[xscale=2.7,yscale=-1.2]
    \node (A0_0) at (0, 0) {$\overline{M}_{g,n}$};
    \node (A0_1) at (1, 0) {$\overline{M}_{g,A[n]}$};
    \node (A1_0) at (0, 1) {$\overline{M}_{g,n-1}$};
    \node (A1_1) at (1, 1) {$\overline{M}_{g,A_3[n-1]}$};
    \path (A0_0) edge [->]node [auto] {$\scriptstyle{\rho}$} (A0_1);
    \path (A1_0) edge [->]node [auto] {$\scriptstyle{\tilde{\rho}}$} (A1_1);
    \path (A0_1) edge [->]node [auto] {$\scriptstyle{\pi_{3}}$} (A1_1);
    \path (A0_0) edge [->,swap]node [auto] {$\scriptstyle{\tilde{\pi}_{3}}$} (A1_0);
  \end{tikzpicture} 
\]
where $\rho$ and $\tilde{\rho}$ are reduction morphisms. Note that even when $a_1+a_2\leq 1$ since just two points are collapsed we have $\tilde{\rho}^{-1}([C^{'},x_1,x_2,x_4,...,x_n]) = [C\cup \mathbb{P}^1,x_1,x_2,x_4,...,x_n]$ and $\tilde{\rho}$ is an isomorphism in a neighborhood of $[C\cup \mathbb{P}^1,x_1,x_2,x_4,...,x_n]$. Therefore $\rho$ defines an isomorphism between $\tilde{\pi}_3^{-1}([C\cup \mathbb{P}^1,x_1,x_2,x_4,...,x_n])\cong C\cup \mathbb{P}^1$ and $\pi_3^{-1}([C^{'},x_1,x_2,x_4,...,x_n])$.\\
The point $\pi_3(x_2)$ is a general point on the divisor $\Delta_{0,\{1,2\}}\subset\overline{M}_{g,A_3[n-1]}$. Therefore the rational map $\tilde{\phi}_3$ is well defined on $\pi_3(x_2)$, and $\phi$ defines an isomorphism between $\pi_3^{-1}(\pi_3(x_2))\cong C\cup \mathbb{P}^1$ and $\pi_3^{-1}([\Gamma,z_1,y_2,y_4,...,y_n])\cong \Gamma$ which is irreducible. A contradiction.
\item[(\textit{b})] Assume that $\phi(x_2)=y_r$ for some $r\in\{3,...,n\}$. Recall that $y_r$ represents the class of a curve $[\Gamma^{'},y_1,...,y_n]$, where either $y_1$ and $y_r$ are collapsed and $\Gamma^{'}\cong \Gamma$ is irreducible or $y_1\neq y_r$ and $\Gamma^{'}\cong \Gamma\cup \mathbb{P}^1$ with $y_1,y_r$ on the rational tail, again depending on whether $a_1+a_r\leq 1$ or $a_1+a_r> 1$. In both cases $\pi_r(y_r)$ is the class $[\Gamma,y_1,y_2,...,\hat{y}_r,...,y_n]$. Now, to get a contradiction it is enough to argue as in point $(a)$ by observing that $\phi$ induces an isomorphism between the reducible curve $\pi_r^{-1}(\pi_r(x_2))$ and the irreducible curve $\pi_r^{-1}(\pi_r(y_r))$.
\end{itemize}
We conclude that $\phi(x_2) = y_2$. Furthermore, the same arguments shows that $\phi(x_i) = y_i$ for any $i=2,...,n$. This means that $\phi_{|F_{1,x}}:F_{1,x}\cong C\rightarrow F_{1,y}\cong \Gamma$ yields an isomorphisms between $C$ and $\Gamma$ mapping $x_i$ to $y_i$ for $i = 2,...,n$. Therefore, $[C,x_2,...,x_n] = [\Gamma,y_2,...,y_n]$ and $\phi$ restricts to an automorphism of the general fiber of $\pi_1$ fixing $n-1$ points. Our numerical hypothesis forces $\phi$ to be the identity on the general fiber of $\pi_1$, and by Remark \ref{autfl} we conclude that $\phi$ is the identity on $\overline{M}_{g,A[n]}$.\\
Now let $\phi\in \Aut(\overline{M}_{g,A[n]})$ and $\sigma_{\phi} = \chi_{g,A[n]}(\phi)$. In the same notation of the previous part of the proof we have that $\phi$ maps $F_{i,x}$ to $F_{\sigma_{\phi}(i),y}$. A priori the permutation $\sigma_{\phi}^{-1}$ induces just a birational automorphism of $\Aut(\overline{M}_{g,A[n]})$. Nevertheless, $\sigma_{\phi}^{-1}\circ \phi$ restricts to a birational automorphism of $F_{1,x}$ fixing $n-1$ points. Now, since $F_{1,x}$ is a smooth curve, the map $\sigma_{\phi}^{-1}\circ \phi$ restricts to an automorphism of $F_{1,x}$ fixing $n-1$ points. Then the birational map $\sigma_{\phi}^{-1}\circ \phi$ restricts to the identity in the general fiber of $\pi_1$. Therefore, $\sigma_{\phi}^{-1}\circ \phi$ is the identity on the open subset of $U\subseteq \overline{M}_{g,A[n]}$ where $\sigma_{\phi}^{-1}$ is defined. We conclude that $\phi$ acts as $\sigma_{\phi}$ on $U$.
Now, the permutation $\sigma_{\phi}$ induces an automorphism of $\overline{M}_{g,n}$. Let us consider the following diagram 
  \[
  \begin{tikzpicture}[xscale=2.7,yscale=-1.2]
    \node (A0_0) at (0, 0) {$\overline{M}_{g,n}$};
    \node (A0_1) at (1, 0) {$\overline{M}_{g,n}$};
    \node (A1_0) at (0, 1) {$\overline{M}_{g,A[n]}$};
    \node (A1_1) at (1, 1) {$\overline{M}_{g,A[n]}$};
    \path (A0_0) edge [->]node [auto] {$\scriptstyle{\sigma_{\phi}}$} (A0_1);
    \path (A1_0) edge [->]node [auto] {$\scriptstyle{\phi}$} (A1_1);
    \path (A0_1) edge [->]node [auto] {$\scriptstyle{\rho}$} (A1_1);
    \path (A0_0) edge [->,swap]node [auto] {$\scriptstyle{\rho}$} (A1_0);
  \end{tikzpicture}
  \]
where $\rho$ is the reduction morphism. A priori this diagram is commutative on the open subset $V = \rho^{-1}(U) \subseteq \overline{M}_{g,n}$. However, since $\rho\circ \sigma_{\phi}$ and $\phi\circ \rho$ are both morphisms coinciding on an open subset we conclude that the diagram is indeed commutative. Therefore $\phi$ must act as the permutation $\sigma_{\phi}$ on the whole of $\overline{M}_{g,A[n]}$, and in turns $\sigma_{\phi}$ induces an automorphism of $\overline{M}_{g,A[n]}$.
\end{proof}

We state the following lemma for the reader's convenience.
\begin{Lemma}\label{s2lemma}
Let $\phi:X\rightarrow Y$ be a continuous map of separated schemes defining a morphism in codimension at least two. If $X$ is $S_{2}$ then $\phi$ is a morphism.
\end{Lemma}
\begin{proof}Let $\mathcal{U}\subset X$ be an open set, whose complement has codimension at least two, where $\phi$ is a morphism. Let $f$ be a regular function on an open subset $\mathcal{V}\subseteq Y$, then  $f\circ\phi_{|\phi^{-1}(\mathcal{V})\cap\mathcal{U}}\in\mathcal{O}_{X}(\phi^{-1}(\mathcal{V})\cap\mathcal{U})$ is a regular function on $\phi^{-1}(\mathcal{V})\cap\mathcal{U}$. Since $X$ is $S_{2}$ $f\circ\phi_{|\phi^{-1}(\mathcal{V})\cap\mathcal{U}}$ extends to a regular function on $X$. So we get a morphism of sheaves $\mathcal{O}_{Y}\rightarrow\phi_{*}\mathcal{O}_{X}$ and $\phi:X\rightarrow Y$ is a morphism of schemes.
\end{proof}
Any transposition $i\leftrightarrow j$ in $S_{n}$ defines a birational map $\tilde{\phi}_{i,j}:\overline{M}_{g,A[n]}\dasharrow\overline{M}_{g,A[n]}$. We aim to understand when this map is an automorphism. Our main tool is the following proposition.
\begin{Proposition}\label{atra}
The following are equivalent:
\begin{itemize}
\item[$(a)$] $i\leftrightarrow j$ is admissible,
\item[$(b)$] $\tilde{\phi}_{i,j}$ is an automorphism.
\item[$(c)$] $\overline{M}_{g,A_{i}[n-1]}\cong\overline{M}_{g,A_{j}[n-1]}$, where  $A_{i} = \{a_{1},...,\hat{a}_{i},...,a_{n}\}$ and $A_{j} = \{a_{1},...,\hat{a}_{j},...,a_{n}\}$, when they are defined.
\end{itemize}
\end{Proposition}
\begin{proof}
$(a)\Rightarrow (b)$ By \cite[Theorem 4.1]{Ha} we have a birational reduction morphism 
$$\rho:\overline{M}_{g,n}\rightarrow\overline{M}_{g,A[n]}.$$
Let $\phi_{i,j}\in\Aut(\overline{M}_{g,n})$ be the automorphism induced by the transposition $i\leftrightarrow j$. Then we have a commutative diagram
\[
  \begin{tikzpicture}[xscale=2.5,yscale=-1.2]
    \node (A0_0) at (0, 0) {$\overline{M}_{g,n}$};
    \node (A0_1) at (1, 0) {$\overline{M}_{g,n}$};
    \node (A1_0) at (0, 1) {$\overline{M}_{g,A[n]}$};
    \node (A1_1) at (1, 1) {$\overline{M}_{g,A[n]}$};
    \path (A0_0) edge [->]node [auto] {$\scriptstyle{\phi_{i,j}}$} (A0_1);
    \path (A1_0) edge [->,dashed]node [auto] {$\scriptstyle{\tilde{\phi}_{i,j}}$} (A1_1);
    \path (A0_1) edge [->]node [auto] {$\scriptstyle{\rho}$} (A1_1);
    \path (A0_0) edge [->]node [auto,swap] {$\scriptstyle{\rho}$} (A1_0);
  \end{tikzpicture}
  \]
where a priori $\tilde{\phi}_{i,j}$ is just a birational map. By \cite[Proposition 4.5]{Ha} $\rho$ contracts the divisors $\Delta_{0,I}$ whose general points correspond to curves with two irreducible components, a genus zero smooth curve with $I = \{i_{1},...,i_{r}\}$ as marking set and a genus $g$ curve with marking set $J = \{1,...,n\}\setminus I = \{j_{1},...,j_{n-r}\}$, such that $a_{i_{1}}+...+a_{i_{r}}\leq 1$ and $2<r\leq n$. A priori $\tilde{\phi}_{i,j}$ is defined just on the open subset of $\overline{M}_{g,A[n]}$ parametrizing curves where each $x_{i},x_{j}$ coincide at most with another marked point. Let $\mathcal{U}\subset\overline{M}_{g,A[n]}$ be the open subset parametrizing such curves.\\ 
Let us consider a curve $[C,x_{1},...,x_{i},...,x_{j},...,x_{n}]$ with $x_{i} = x_{i_{2}} = ... = x_{i_{r}}$, $2<r\leq n-1$. By Definition \ref{defha} we have $a_{i}+a_{i_{2}}+...+a_{i_{r}}\leq 1$. Then $\rho^{-1}([C,x_{1},...,x_{i},...,x_{j},...,x_{n}])$ lies on a divisor of type $\Delta_{I,J}$. By Definition \ref{atrans} we have $a_{j}+a_{i_{2}}+...+a_{i_{r}}\leq 1$. So $(\rho\circ\phi_{i,j}\circ\rho^{-1})([C,x_{1},...,x_{i},...,x_{j},...,x_{n}]) = [C,x_{1},...,x_{j},...,x_{i},...,x_{n}]$ with $x_{j} = x_{i_{2}} = ... = x_{i_{r}}$. We consider the same construction for curves $[C,x_{1},...,x_{i},...,x_{j},...,x_{n}]$ with $x_{j} = x_{i_{2}} = ... = x_{i_{r}}$, $2<r\leq n-1$ and extend $\tilde{\phi}_{i,j}$ as a continuous map by 
$$\tilde{\phi}_{i,j}([C,x_{1},...,x_{i},...,x_{j},...,x_{n}]):= [C,x_{1},...,x_{j},...,x_{i},...,x_{n}].$$
The continuous map $\tilde{\phi}_{i,j}:\overline{M}_{g,A[n]}\rightarrow\overline{M}_{g,A[n]}$ defines an isomorphism from an open subset $\mathcal{U}$ to itself. Furthermore, the complement of $\mathcal{U}$ has codimension at least two. This is enough to conclude, by Remark \ref{normality}  and Lemma \ref{s2lemma}, that $\tilde{\phi}_{i,j}$ is an isomorphism.\\
$(b)\Rightarrow (c)$ This is a particular case of Lemma \ref{lsal}.\\
$(c)\Rightarrow (a)$ We may assume that $a_i\leq a_j$. Then, by \cite[Proposition 4.5]{Ha}, the reduction morphism $\rho_{A_i[n-1],A_j[n-1]}:\overline{M}_{g,A_{j}[n-1]}\to\overline{M}_{g,A_{i}[n-1]}$ is an isomorphism. Therefore, again by \cite[Proposition 4.5]{Ha},
$a_{j}+\sum_{k = 1}^{r}a_{h_{k}}\leq 1$ and $a_{i}+\sum_{k = 1}^{r}a_{h_{k}}>1$ is possible only if $r\leq 1$.  This shows that $i\leftrightarrow j$ is admissible.
\end{proof}
The following definition is the key to describe the image of $\chi$.
\begin{Definition}\label{def:chi_admissible} 
Let $ \overline{M}_{g,A_{[n]}}$ be a Hassett's space with a morphism $\chi_{g,A[n]}:\Aut(\overline{M}_{g,A_{[n]}})\rightarrow S_m$ as in (\ref{eq:chi}). Then $\mathcal{S}_{A[n]}\subseteq S_m$ is the subgroup generated by admissible transpositions.
\end{Definition}
We are ready to compute the image of $\chi_{g,A[n]}$.

\begin{Lemma}\label{surj}
If $2g-2+n\geq 3$ and in Equation (\ref{eq:chi}) we have $m = n$ then the subgroup $\mathcal{S}_{A[n]}\subset S_m$ is the image of $\chi_{g,A[n]}$ and 
$$\Aut(\overline{M}_{g,A_{[n]}})\iso \mathcal{S} _{A[n]}.$$
\end{Lemma}
\begin{proof}
By Remark~\ref{n<3} and \cite{Ma} we may assume that $n\geq 3$.
By Proposition \ref{atra} we have $\mathcal{S}_{A[n]}\subset\im( \chi_{g,A[n]})$. Let $\sigma_{\phi} = \chi(\phi)$ be the permutation induced by $\phi\in\Aut(\overline{M}_{g,A[n]})$. By Proposition \ref{k2015} $\ker(\chi_{g,A[n]})$ is trivial, $\sigma_{\phi}$ induces an automorphism of $\overline{M}_{g,A[n]}$, and $\phi$ acts on $\overline{M}_{g,A[n]}$ as the permutation $\sigma_{\phi}$.\\
Consider a cycle $\sigma_{i_1,...,i_r}=(i_{1}...i_{r})$ in $\sigma_{\phi}$ and its decomposition $(i_{1}i_{r})(i_{1}i_{r-1})...(i_{1}i_{3})(i_{1}i_{2})$ as product of transpositions. Since $\sigma_{i_1,...,i_r}$ induces an automorphism and $\sigma_{i_1,...,i_r}(i_{j-1}) =i_j$. By Lemma \ref{lsal} applied to the permutation $\sigma_{i_1,...,i_r}$ for any $j = 2,...,r$ we have the following commutative diagram: 
  \[
  \begin{tikzpicture}[xscale=3.5,yscale=-1.2]
    \node (A0_0) at (0, 0) {$\overline{M}_{g,A[n]}$};
    \node (A0_1) at (1, 0) {$\overline{M}_{g,A[n]}$};
    \node (A1_0) at (0, 1) {$\overline{M}_{g,A_{j-1}[n-1]}$};
    \node (A1_1) at (1, 1) {$\overline{M}_{g,A_{j}[n-1]}$};
    \path (A0_0) edge [->]node [auto] {$\scriptstyle{\sigma_{i_1,...,i_r}}$} (A0_1);
    \path (A1_0) edge [->]node [auto] {$\scriptstyle{\tilde{\sigma}_{i_1,...,i_r}}$} (A1_1);
    \path (A0_1) edge [->]node [auto] {$\scriptstyle{\pi_{i_j}}$} (A1_1);
    \path (A0_0) edge [->,swap]node [auto] {$\scriptstyle{\pi_{i_{j-1}}}$} (A1_0);
  \end{tikzpicture}
  \]
where $\tilde{\sigma}_{i_1,...,i_r}:\overline{M}_{g,A_{j-1}[n-1]}\rightarrow\overline{M}_{g,A_{j}[n-1]}$ is an isomorphism. Therefore, we get
$$\overline{M}_{g,A_{1}[n-1]}\cong\overline{M}_{g,A_{2}[n-1]}\cong ...\cong\overline{M}_{g,A_{j-1}[n-1]}\cong \overline{M}_{g,A_{j}[n-1]}\cong ...\cong \overline{M}_{g,A_{r-1}[n-1]}\cong\overline{M}_{g,A_{r}[n-1]}.$$
In particular, $\overline{M}_{g,A_{1}[n-1]}\cong \overline{M}_{g,A_{j}[n-1]}$ for any $j=2,...,r$. Then, by Proposition \ref{atra}, the transposition $(i_{1}i_{j})$ is admissible for any $j = 2,...,r$, and $\sigma_{\phi}\in \mathcal{S}_{A[n]}$.
\end{proof}

The following Lemma is well known and can be checked via a direct computation, see also \cite{GP}.
\begin{Lemma}\label{semi-direct}
Let $H_{n-3,Cr}$ be the group generated by the torus $(\mathbb{C}^{*})^{n-3}$ of automorphisms of $\mathbb{P}^{n-3}$ fixing the toric invariant points and by the standard Cremona transformation $Cr:\mathbb{P}^{n-3}\dasharrow\mathbb{P}^{n-3}$. Then 
$$H_{n-3,Cr}\cong (\mathbb{C}^{*})^{n-3}\rtimes S_2$$
where $S_2 =\{Id,Cr\}$, and the semi-direct product is not direct.\\
Now, let $S_{n-2}$ act on $\mathbb{P}^{n-3}$ by permuting the homogeneous coordinates. Then the action of $S_{n-2}$ commutes with the action of $S_2$ but does not commute with the action of $(\mathbb{C}^{*})^{n-3}$. Finally $(\mathbb{C}^{*})^{n-3}$ is a normal subgroup of the group generated by $H_{n-3,Cr}$ and $S_{n-2}$.
\end{Lemma}

We are ready to compute the automorphism group. Let us start with the  spaces $\overline{M}_{0,A_{r,s}[n]}$. To help the reader in getting acquainted with the ideas of the proof we describe in detail the case $n=6$, where all issues appear. Construction \ref{kblu} is as follows:
\begin{itemize}
\item[-] $r = 1$, $s = 1$, gives $\mathbb{P}^{3}$, 
\item[-] $r = 1$, $s = 2$, we blow-up the points $p_{1},...,p_{4}\in\mathbb{P}^{3}$ and get the Hassett's space with weights $A_{1,2}[6]:= (1/4,1/4,1/4,1/4,1/2,1)$,
\item[-] $r = 1$, $s = 3$, we blow-up the strict transforms of the lines $\left\langle p_{i},p_{j}\right\rangle$, $i,j=1,...,4$, and get the Hassett's space with weights $A_{1,3}[6]:= (1/4,1/4,1/4,1/4,3/4,1)$,
\item[-] $r = 2$, $s = 1$, we blow-up the point $p_{5}$, and get the Hassett's space with weights $A_{2,1}[6]:= (1/3,1/3,1/3,1/3,1,1)$,
\item[-] $r = 2$, $s = 2$, we blow-up the strict transforms of the lines $\left\langle p_{i},p_{5}\right\rangle$, $i=1,...,3$, and get the Hassett's space with weights $A_{2,2}[6]:= (1/3,1/3,1/3,2/3,1,1)$,
\item[-] $r = 3$, $s = 1$, we blow-up the strict transform of the line $\left\langle p_{4},p_{5}\right\rangle$ and get the Hassett's space with weights $A_{3,1}[6]:= (1/2,1/2,1/2,1,1,1)$, that is $\overline{M}_{0,6}$.
\end{itemize}
\begin{Proposition}\label{n6}
If $n = 6$ we have:
\begin{itemize}
\item[-] $\Aut(\overline{M}_{0,A_{1,3}[6]})\cong (\mathbb{C}^{*})^{3}\rtimes (S_2\times S_{4})$,
\item[-] $\Aut(\overline{M}_{0,A_{2,1}[6]})\cong  S_4\times S_2$,
\item[-] $\Aut(\overline{M}_{0,A_{2,2}[6]})\cong S_3\times S_3$,
\item[-] $\Aut(\overline{M}_{0,A_{3,1}[6]})\cong S_6$
\end{itemize} 
\end{Proposition}
\begin{proof} In the notations of Construction \ref{kblu} let $\rho:\overline{M}_{0,A_{r,s}[6]}\to\P^3$ be a reduction morphism factorizing the morphism $f_6:\overline{M}_{0,6}\rightarrow\mathbb{P}^3$. 
Let us start with  $r = 1, s = 3$. Then, as in  Equation (\ref{mor1}), we have the morphism
$$\chi_{1,3}:\Aut(\overline{M}_{0,A_{1,3}[6]})\rightarrow \mathcal{S}_{A_{2,1}[6]}\subseteq S_4.$$
The first four weights are equal therefore $\chi_{1,3}$ is surjective. An automorphism $\phi$ of $\overline{M}_{0,A_{r,s}[6]}$ whose image in $S_{4}$ is the identity induces, via $\rho$, a birational transformation $\phi_{\H}:\mathbb{P}^{3}\dasharrow\mathbb{P}^{3}$ stabilizing the lines through $p_{i}$ for $i=1,2,3,4$. Let $\mathcal{H}\subseteq |\mathcal{O}_{\mathbb{P}^{3}}(d)|$ be the linear system associated to $\phi_{\H}$. If $L_{i}$ is a line through $p_{i}$ we have
$$
\begin{tabular}{l}
$\deg(\phi_{\H}(L_{i})) = d - \mult_{p_{i}}\mathcal{H} = 1.$
\end{tabular} 
$$
This yields 
\begin{equation}\label{eq:3d}
\mult_{p_{i}}\mathcal{H} = d-1,\ \mult_{\langle p_i,p_j\rangle}\mathcal{H}\geq d-2,\mbox{\rm and}\ \mult_{\langle p_i,p_j,p_k\rangle}\mathcal{H}\geq d-3. 
\end{equation}
The linear system  $\mathcal{H}$ does not have fixed components therefore $d\leq 3$ and in (\ref{eq:3d}) all inequalities are equalities.
If $d = 1$ then $\phi_{\H}$ is an automorphism of $\mathbb{P}^{3}$ fixing $p_{1},p_{2},p_{3},p_{4}$. These correspond to diagonal, non-singular matrices.\\
If $d\neq 1$,  again by Theorem \ref{facg0}, we have the following commutative diagram 
 \[
  \begin{tikzpicture}[xscale=2.5,yscale=-1.2]
    \node (A0_0) at (0, 0) {$\overline{M}_{0,A[n]}$};
   \node (A0_1) at (1, 0) {$\overline{M}_{0,A[n]}$};
    \node (A1_0) at (0, 1) {$\overline{M}_{0,A[n-2]}$};
    \node (A1_1) at (1, 1) {$\overline{M}_{0,A[n-2]}$};
    \path (A0_0) edge [->]node [auto] {$\scriptstyle{\phi^{-1}}$} (A0_1);
   \path (A1_0) edge [->,dashed]node [auto] {$\scriptstyle{\tilde{\phi}}$} (A1_1);
    \path (A0_1) edge [->]node [auto] {$\scriptstyle{\pi_{i_1,i_2}}$} (A1_1);
    \path (A0_0) edge [->]node [auto,swap] {$\scriptstyle{\pi_{j_i,j_2}}$} (A1_0);
  \end{tikzpicture}
  \]
Therefore, $\phi_{\H}$ induces a Cremona transformation on the general plane containing the line $\left\langle p_{1},p_{2}\right\rangle$. So on such a general plane the linear system $\mathcal{H}$ needs a third base point, outside $\langle p_{1},p_{2}\rangle$. This means that in $\mathbb{P}^{3}$ a codimension two linear space has to be blown-up. So $s = d = 3$ and $\phi_{\H}$ is the standard Cremona transformation of $\mathbb{P}^{3}$. Kapranov proved \cite[Proposition 2.12]{Ka} that this automorphism corresponds to the transposition of the last two marked points. We conclude that $\Ker(\chi_{1,3}) = H_{3,Cr}$. 
Now, by Lemma \ref{semi-direct} the action of $S_4$ commutes with the action of $S_2$ but does not with the action of $(\mathbb{C}^{*})^{3}$. By Lemma \ref{semi-direct} $H_{3,Cr}\cong (\mathbb{C}^{*})^{3}\rtimes S_2$ and $(\mathbb{C}^{*})^{3}$ is normal in $H_{3,Cr}\rtimes S_{4}$. This yields $\Aut(\overline{M}_{0,A_{1,3}[6]})\cong H_{3,Cr}\rtimes S_{4}\cong (\mathbb{C}^{*})^{3}\rtimes (S_2\times S_{4}) $.\\
Now, let us consider the case $r = 2, s=1$. There is a morphism (\ref{mor2}):
$$\chi_{2,1}:\Aut(\overline{M}_{0,A_{2,1}[6]})\rightarrow S_4.$$ 
The first four weights are equal therefore $\chi_{2,1}$ is surjective. This moduli space is realized after the blow-up of the point $p_5$. Therefore, the only linear automorphism that lifts is the identity.  Arguing as in the previous case we are left with $\phi_\H$ the standard Cremona transformation. We may assume $p_1=[1:0:0:0],...,p_4 = [0:0:0:1]$ and $p_5 = [1:1:1:1]$. The standard Cremona is given by:
$$
\begin{array}{cccc}
Cr: & \mathbb{P}^3 & \dasharrow & \mathbb{P}^3\\
 & [x_0:x_1:x_2:x_3] & \longmapsto & [\frac{1}{x_0}:\frac{1}{x_1}:\frac{1}{x_2}:\frac{1}{x_3}]
\end{array}
$$
and $Cr(p_5) = p_5$. Therefore, $\phi_{\H} = Cr$ lifts to an automorphism of $\overline{M}_{0,A_{2,1}[6]}$ contained in $\Ker(\chi_{2,1})$. Kapranov proved \cite[Proposition 2.12]{Ka} that this automorphism corresponds to the transposition of the last two marked points. We conclude that $\Ker(\chi_{2,1}) = S_2$. Again the actions of $S_4$ and of the Cremona commute. This yields a direct product $\Aut(\overline{M}_{0,A_{2,1}[6]})\cong  S_4\times S_2$.\\
When either $r = 2$ and $s = 2$ or $r=3$, by Equation (\ref{mor3}) and Lemma~\ref{surj} there is an isomorphism
$$\chi_{r,s}:\Aut(\overline{M}_{0,A_{r,s}[6]})\rightarrow \mathcal{S}_{A_{r,s}[6]}\subseteq S_6.$$
It is a simple check to see that $A_{2,2}[6]=S_3\times S_3$, while $A_{3,1}[6]=S_6$.
\end{proof}

Now, let us consider the general case. The following lemma generalizes the ideas in the proof of Proposition \ref{n6} and leads us to control the degree and type of linear systems involved in the computation of the automorphisms of the spaces appearing in Construction \ref{kblu}.
\begin{Lemma}\label{kl}
Let $\H\subset|\O_{\P^{n-3}}(d)|$ be a linear system, with $d>1$, and $\{p_1,\ldots,p_a\}\subset\P^{n-3}$ a collection of points. Assume that $\mult_{p_i}\H=d-1$, for $i=1,\ldots,a$. Let $L_{i_1,\ldots,i_h}=\Span{p_{i_1},\ldots,p_{i_h}}$ be the linear span of $h$ points in $\{p_1,\ldots,p_a\}$, then 
$$\mult_{L_{i_1,\ldots,i_h}} \H\geq d-h.$$
Assume further that $\H$ does not have fixed components, $a=n-2$, $\H$ preserves the lines through the points $p_i$, and the rational map, say $\f_\H$, induced by $\H$ lifts to an automorphism of  $\overline{M}_{0,{A}_{r,s}[n]}$. Then 
$$\mult_{L_{i_1,\ldots,i_h}} \H= d-h,$$
and either $r\geq 2$ or $r=1$ and $s = d = n-3$. Furthermore $\f_\H$ is the standard Cremona transformation centred at $\{p_1,\ldots,p_{n-2}\}$.
\end{Lemma}
\begin{proof}
The first statement is meaningful only for $h<d$. We prove it by a double induction on $d$ and $h$. The initial case $d=2$ and $a=1$ is immediate.  Let us consider $\Pi:=L_{p_{i_1},\ldots,p_{i_h}}$ and $L_j=\Span{p_{i_1},\ldots,\hat{p}_{i_j},...,p_{i_h}}$ the linear span of $h-1$ points in $\{p_1,\ldots,p_h\}$. Then by induction hypothesis
$$\mult_{L_j}\H_{|\Pi}\geq d-(h-1),$$ 
and  $L_{j}$ is a divisor in $\Pi$. By assumption $d>h$ hence $d(h-1)>h(h-1)$ and
$$h(d-(h-1))>d.$$
This yields $\Pi\subset\Bl\H$. Let $A$ be a general linear space of dimension $h$ containing $\Pi$. Then we may decompose $\H_{|A}=\Pi+\H_1$ with $\H_1\subset|\O(d-1)|$ and
\begin{equation}
\mult_{L_j}\H_1 \geq d-1-(h-1).
\label{eq:mult}
\end{equation}
Arguing as above this forces $\Pi\subset\H_1$ as long as $h(d-1-(h-1))>d-1$, that is $d-1>h$, and recursively gives the first statement.

Assume that $r=1$ and the map $\f_\H$ lifts to an automorphism of  $\overline{M}_{0,{A}_{1,s}[n]}$. This forces some immediate consequences:
\begin{itemize}
\item[$(a)$] the points $p_i$ are in general position,
\item[$(b)$] the scheme theoretic base locus of $\H$ is the span of all subsets of at most $s-1$ points.
\end{itemize}
Since $L_{p_{i_1},\ldots,p_{i_s}}\not\subset\Bl(\H)$ equation (\ref{eq:mult}) yields
\begin{equation}\label{eq:ub}
s\geq d.
\end{equation}
Furthermore, the hyperplane $H = \Span{p_{i_{1}},...,p_{i_{n-3}}}$ contains $(n-3)$ codimension two linear spaces of the form $L_{j}$, each of multiplicity $d-(n-4)$ for the linear system $\H$. The linear system $\H$ does not have fixed components hence $(n-3)(d-n+4)\leq d$ and we get
$$d\leq n-3.$$
\begin{claim}
$\Bl\H\not\supset L_{i_1,\ldots,i_d}$.
\end{claim}
\begin{proof}
Assume that $\Bl\H\supset L_{i_1,\ldots,i_d}$. Then the restriction $\H_{|L_{i_1,\ldots,i_{d+1}}}$ contains a fixed divisor of degree $d+1$ and $L_{i_1,\ldots,i_{d+1}}\subset\Bl\H$. A recursive argument then shows that $\Bl\H$ has to contain all the linear spaces spanned by the $n-2$ points yielding a contradiction.
\end{proof}
The claim above together with $(b)$ and equation (\ref{eq:ub}) yield
$$s=d,$$ 
and
$$\mult_{L_{i_1,\ldots,i_{d-1}}}\H = d-(d-1) = 1.$$
Then, recursively this forces the equality in equation (\ref{eq:mult}) for any value of $h$. To conclude let us consider the commutative diagram
\[
  \begin{tikzpicture}[xscale=2.5,yscale=-1.2]
    \node (A0_0) at (0, 0) {$\overline{M}_{0,A_{1,s}[n]}$};
    \node (A0_1) at (1, 0) {$\overline{M}_{0,A_{1,s}[n]}$};
    \node (A1_0) at (0, 1) {$\mathbb{P}^{n-3}$};
    \node (A1_1) at (1, 1) {$\mathbb{P}^{n-3}$};
    \node (A2_0) at (0, 2) {$\mathbb{P}^{n-5}$};
    \node (A2_1) at (1, 2) {$\mathbb{P}^{n-5}$};
    \path (A0_0) edge [->]node [auto] {$\scriptstyle{\phi}$} (A0_1);
    \path (A2_0) edge [->,dashed]node [auto] {$\scriptstyle{}$} (A2_1);
    \path (A1_0) edge [->,dashed]node [auto] {$\scriptstyle{\phi_{\mathcal{H}}}$} (A1_1);
    \path (A1_0) edge [->,dashed]node [auto] {$\scriptstyle{}$} (A2_0);
    \path (A1_1) edge [->,dashed]node [auto] {$\scriptstyle{}$} (A2_1);
    \path (A0_0) edge [->]node [auto] {$\scriptstyle{}$} (A1_0);
    \path (A0_1) edge [->]node [auto] {$\scriptstyle{}$} (A1_1);
  \end{tikzpicture}
  \]
By Theorem \ref{facg0} we know that $\f$ composed with a forgetful map onto $\overline{M}_{0,{A}_{1,s}[n-2]}$ is again a forgetful map. This forces the map $\f_\H$ to induce a Cremona transformation on the general plane containing $\{p_{i_1},p_{i_2}\}$. Let $\Pi$ be a general plane containing $\{p_{i_1},p_{i_2}\}$. Then the mobile part of $\H_{|\Pi}$ is a linear system of conics with two simple base points in $p_{i_1}$ and $p_{i_2}$. This forces the presence of a further base point to produce a Cremona transformation. Therefore a codimension two linear space has to be blown-up. This shows that $s=d=n-3$. To conclude we observe that the linear system $\mathcal{H}$ of forms of degree $n-3$ in $\P^{n-3}$ such that $\mult_{L_{i_1,...,i_h}}\mathcal{H} = n-h-3$ for $h = 1,...,n-4$, has dimension $n-2$ and gives rise to the standard Cremona transformation.

Assume that $r\geq 2$. Then, by the above argument, $\H$ is a sublinear system, of projective dimension at least $n-3$, of the linear system inducing the standard Cremona transformation therefore $\H$ is the linear system of the standard Cremona transformation.
\end{proof}
\begin{Theorem}\label{authaka}
For the Hassett's spaces appearing in Construction \ref{kblu} we have:
\begin{itemize}
\item[-] $\Aut(\overline{M}_{0,A_{1,n-3}[n]})\cong (\mathbb{C}^{*})^{n-3}\rtimes (S_2\times S_{n-2})$,
\item[-] $\Aut(\overline{M}_{0,A_{2,1}[n]})\cong S_{n-2}\times S_2$, 
\item[-] $\Aut(\overline{M}_{0,A_{2,s}[n]})\cong S_{n-3}\times S_2$, if $1< s <n-4$, 
\item[-] $\Aut(\overline{M}_{0,A_{r,s}[n]})\cong \mathcal{S}_{A_{r,s}[n]}$, if either $r = 2$ and $s=n-4$, or\: $r \geq 3$,
\end{itemize}
\end{Theorem}
\begin{proof} In the notations of Construction \ref{kblu} let $\rho:\overline{M}_{0,A_{r,s}[n]}\to\P^{n-3}$ be a reduction morphism factorizing the morphism $f_n:\overline{M}_{0,n}\rightarrow\mathbb{P}^{n-3}$. Let us begin with the case $r = 1$, $s = n-3$, then, by  Equation (\ref{mor1}), we have a morphism
$$\chi_{1,n-3}:\Aut(\overline{M}_{0,A_{1,n-3}[n]})\rightarrow S_{n-2}.$$
Furthermore, since the first $n-2$ weights are all equal we have that $\chi_{1,n-3}$ is surjective. Let $\phi$ be an automorphism of $\overline{M}_{0,A_{1,n-3}[n]}$ such that $\chi_{1,n-3}(\phi)$ is the identity. Then $\phi$ preserves the forgetful maps onto $\overline{M}_{0,{A}_{1,n-3}[n-1]}$ and the birational map $\phi_{\H}$ induced by $\phi$ stabilizes the lines through $p_{1},...,p_{n-2}$.\\
Let $\mathcal{H}\subseteq |\mathcal{O}_{\mathbb{P}^{n-3}}(d)|$ be the linear system associated to $\phi_{\H}$. If $L_{i}$ is a line through $p_{i}$ we have
$$\deg(\phi_{\H}(L_{i})) = d - \mult_{p_{i}}\mathcal{H} = 1.$$
So $\mult_{p_{i}}\mathcal{H} = d-1$. If $d=1$ then $\phi_\H$ is a linear automorphism that fixes $(n-2)$ points and corresponds to the diagonal matrices. If $d>1$ then by Lemma \ref{kl}, the linear system $\H$ is associated to the standard Cremona transformation of $\mathbb{P}^{n-3}$, and by Kapranov \cite[Proposition 2.12]{Ka}, induces the transposition $n-1\leftrightarrow n$. This gives $\Ker(\chi_{1,n-3}) = H_{n-3,Cr}$. By Lemma \ref{semi-direct} the action of $S_{n-2}$ commutes with $S_2$ but it does not commute with $(\mathbb{C}^{*})^{n-3}$, and $(\mathbb{C}^{*})^{n-3}$ is normal in $H_{n-3,Cr}\rtimes S_{n-2}$. This yields $\Aut(\overline{M}_{0,A_{1,n-3}[n]})\cong H_{n-3,Cr}\rtimes S_{n-2}\cong (\mathbb{C}^{*})^{n-3}\rtimes (S_2\times S_{n-2})$.\\ 
Now, we consider the case $r = 2$, $s\leq n-5$. Again, by Equation~(\ref{mor2}), we have a morphism
$$\chi_{2,s}:\Aut(\overline{M}_{0,A_{2,s}[n]})\rightarrow S_{n-2}.$$
Recall that the weights are
$$A_{2,s}[n]:= (\underbrace{1/(n-3),...,1/(n-3)}_{(n-3)-\rm{times}}, s/(n-3), 1,1).$$ 
Therefore the transposition $n-2 \leftrightarrow i$ for $i = 1,...,n-3$ induces an automorphism if and only if $s = 1$. This means that $\chi_{2,1}$ is surjective while the image of $\chi_{2,s}$ is $S_{n-3}$ for $s>1$.\\
As before an automorphism $\phi$ of $\overline{M}_{0,A_{2,s}[n]}$ that induces the trivial permutation in $S_{n-2}$ yields a birational transformation $\phi_{\H}:\mathbb{P}^{n-3}\dasharrow\mathbb{P}^{n-3}$ stabilizing the lines through $p_{i}$, for $i=1,2,3,...,n-2$ and it is forced to be either an automorphism of $\mathbb{P}^{n-3}$ or the standard Cremona transformation. We may assume that $p_1=[1:0:...:0],...,p_{n-2} = [0:...:0:1]$ and $p_{n-1} = [1:...:1]$. Due to the blow-up of the point $p_{n-1}$ the only linear transformation that lifts to an automorphism on $\overline{M}_{0,A_{2,s}[n]}$ is the identity. The standard Cremona is given by:
$$
\begin{array}{cccc}
Cr: & \mathbb{P}^{n-3} & \dasharrow & \mathbb{P}^{n-3}\\
 & [x_0:...:x_{n-3}] & \longmapsto & [\frac{1}{x_0}:...:\frac{1}{x_{n-3}}]
\end{array}
$$
Therefore $Cr(p_{n-1}) = p_{n-1}$, and any linear subspace of $\mathbb{P}^{n-3}$ spanned by $p_{n-1}$ and a subset of $\{p_1,...,p_{n-2}\}$ is stabilized by $Cr$. In particular  $\phi_{\H} = Cr$ lifts to an automorphism of $\overline{M}_{0,A_{2,s}[n]}$ contained in $\Ker(\chi_{2,s})$. By \cite[Proposition 2.12]{Ka} this automorphism corresponds to the transposition of the last two marked points. Therefore $\Ker(\chi_{2,s}) = S_2$. When either $r\geq 2$ and $s = n-4$, or $r\geq 3$, by Lemma \ref{surj} there is an isomorphism $\chi_{r,s}:\Aut(\overline{M}_{0,A_{r,s}[n]})\rightarrow \mathcal{S}_{A_{r,s}[n]}\subseteq S_{n}$.
\end{proof}

\begin{Remark}\label{comm}
 For arbitrary weights it is difficult to compute the kernel of $\chi_{0,A[n]}$. We believe that for many  Hassett's spaces $m<n$, the kernel of  $\chi_{0,A[n]}:\Aut(\overline{M}_{0,A[n]})\rightarrow S_m$  is not trivial,  and the actions of $\Ker(\chi_{0,A[n]})$ and $\im( \chi_{0,A[n]})$ on $\overline{M}_{0,A[n]}$ do not commute. 
For instance, consider the Hassett's spaces appearing in \cite[Section 6.3]{Ha}. The space $\overline{M}_{0,A[5]}$ with 
$$A[5] = (1-2\epsilon,1-2\epsilon,1-2\epsilon,\epsilon,\epsilon)$$ 
where $\epsilon$ is an arbitrarily small positive rational number, is isomorphic to $\mathbb{P}^{1}\times\mathbb{P}^{1}$. Since the forgetful morphisms forgetting the fourth and the fifth point are the only fibrations of $\overline{M}_{0,A[5]}$ on $\mathbb{P}^1$ we may still construct a surjective morphism of groups $\chi_{0,A[5]}:\Aut(\overline{M}_{0,A[5]})\rightarrow S_2$ whose kernel is $PGL(2)\times PGL(2)$. In this case $\Aut(\overline{M}_{0,A[5]})\cong \Ker(\chi_{0,A[5]})\rtimes\mathcal{S} _{A[5]}\cong (PGL(2)\times PGL(2))\rtimes S_2$ and the semi-direct product is not direct.
\end{Remark}

To conclude the proof of Theorem \ref{th:aut0} it is enough to compute $\mathcal{S}_{A_{r,s}[n]}$.

\begin{Proposition}\label{sympres}
If $2\leq r\leq n-4$ we have
$$\mathcal{S}_{A_{r,s}[n]}
\iso\left\{\begin{array}{lcl}
S_{n-r}\times S_r & if & r\geq 3,\: s=1,\\
S_{n-r-1}\times S_r & if & r\geq 3,\: 1< s < n-r-2,\\
S_{n-r-1}\times S_{r+1} & if & r\geq 2,\: s = n-r-2.  
\end{array}\right.
$$
If $r = n-3$ then $s =1$, $\overline{M}_{0,A_{n-3,1}[n]}\cong\overline{M}_{0,n}$, and $\mathcal{S}_{A_{n-3,1}[n]}\cong S_n$.
\end{Proposition}
\begin{proof}
Let us begin with the case $2\leq r\leq n-4$. Since $r\leq n-4$ in Equation~(\ref{wdr2}) we have at least three marked points with weight $\frac{1}{n-r-1}$.\\
If $s = 1$ from Equation~(\ref{wdr2}) we see that the first $n-r$ are all equal to $\frac{1}{n-r-1}$ while the last $r$ weights are one. Clearly a transposition between two points of these two sets is not admissible. If $r\geq 3$ we get that the image of the morphism in Equation~(\ref{mor3}) is $\mathcal{S}_{A_{r,1}[n]} = S_{n-r}\times S_r$.\\
Consider the case $1 < s < n-r-2$. Since $s>1$ we have $(n-r-2)\frac{1}{n-r-1}+\frac{s}{n-r-1} >1$. However, the first $n-r-1$ weights sum exactly to one. Therefore, a transposition between a point with weight $\frac{1}{n-r-1}$ and the point with weight $\frac{s}{n-r-1}$ is not admissible. On the other hand, $s< n-r-2$ yields $\frac{s}{n-r-1}+\frac{1}{n-r-1}+\frac{1}{n-r-1}\leq 1$. Therefore, also a transposition between the point with weight $\frac{s}{n-r-1}$ and a point with weight one is not admissible. If $r\geq 3$ the image of the morphism in Equation~(\ref{mor3}) is $\mathcal{S}_{A_{r,s}[n]} = S_{n-r-1}\times S_r$.\\
Finally, if $s = n-r-2$ we have $\frac{s}{n-r-1}+\frac{1}{n-r-1}+\frac{1}{n-r-1}> 1$, and a transposition between the point with weight $\frac{s}{n-r-1}$ and a point with weight one is admissible. Therefore, the image of the morphism in Equation~(\ref{mor3}) is $\mathcal{S}_{A_{r,s}[n]} = S_{n-r-1}\times S_{r+1}$.\\
Now, let us consider the case $r =n-3$. Note that in Equation~(\ref{wdr2}) we have just two marked points with weight $\frac{1}{n-r-1} = \frac{1}{2}$. Therefore, any transposition is admissible.  
\end{proof}

\begin{Remark}
In Remark \ref{LM} we identified the space $\overline{M}_{0,A_{1,n-3}[n]}$ with the Losev-Manin's space $\overline{L}_{n-2}$. By Theorem \ref{authaka}, we have 
$$\Aut(\overline{L}_{n-2})\cong (\mathbb{C}^{*})^{n-3}\rtimes (S_2\times S_{n-2}).$$
From the description of $\overline{L}_{n-2}$ given in Remark \ref{LM} it is clear that $S_{n-2}$ gives the permutations of $x_{1},...,x_{n-2}$ while $S_{2}$ corresponds to the transposition $x_{0}\leftrightarrow x_{\infty}$.\\
In particular, we recover the classical result on the automorphism group of a del Pezzo surface $\mathcal{S}_{6}$ of degree six $\Aut(\mathcal{S}_{6})\cong (\mathbb{C}^{*})^{2}\rtimes (S_2\times S_{3})$ \cite[Chapter 8]{Do}, \cite[Section 6]{DI}. For details on the automorphisms of a simplicial toric variety see \cite[Section 4]{Co}.\\
Furthermore, when $r = n-3$, $s = 1$ we have $\overline{M}_{0,A_{1,n-3}[n]}\cong \overline{M}_{0,n}$. Therefore, since $\mathcal{S}_{A_{n-3,1}[n]} = S_n$ from Theorem \ref{authaka} we recover \cite[Theorem 3]{BM2}, that is $\Aut(\overline{M}_{0,n})\cong S_n$ for any $n\geq 5$. Finally, we would like to stress that by \cite[Section 4.10]{Li1} and \cite[Section 12.3]{Li2} this result holds for $M_{0,n}$ as well, that is $\Aut(M_{0,n})\cong S_n$ for any $n\geq 5$.
\end{Remark}

Next we compute the automorphism groups of the Hassett's spaces $\overline{M}_{g,A[n]}$ with $g\geq 1$.

\begin{Theorem}\label{autg}
If $g\geq 1$ and $2g-2+n\geq 3$ then the automorphism group of $\overline{M}_{g,A[n]}$ is isomorphic to the group of admissible permutations
$$\Aut(\overline{M}_{g,A[n]})\cong \mathcal{S}_{A[n]}.$$
Furthermore, $\Aut(\overline{M}_{1,A[1]})\cong PGL(2)$, and $\Aut(\overline{M}_{1,A[2]})\cong (\mathbb{C}^{*})^2$ for any vector of weights.
\end{Theorem}
\begin{proof}
If  $2g-2+n\geq 3$ the isomorphism $\Aut(\overline{M}_{g,A[n]})\cong \mathcal{S}_{A[n]}$ follows by Lemma \ref{surj}.\\
By Remark~\ref{n<3} the remaining cases are covered by \cite{Ma}. More precisely $\Aut(\overline{M}_{1,2})\cong (\mathbb{C}^{*})^2$, by \cite[Proposition 3.8]{Ma}, and   $\Aut(\overline{M}_{g})\cong\Aut(\overline{M}_{g,1})$ is the trivial group for $g\geq 2$, by \cite[Propositions 3.5, 2.6]{Ma}.
\end{proof}

\subsubsection*{Automorphisms of $\overline{\mathcal{M}}_{g,A[n]}$}\label{stack}
Let us consider the Hassett's moduli stack $\overline{\mathcal{M}}_{g,A[n]}$ and the natural morphism $\pi:\overline{\mathcal{M}}_{g,A[n]}\rightarrow\overline{M}_{g,A[n]}$ on its coarse moduli space. Since $\pi$ is universal for morphism to schemes for any $\phi\in\Aut(\overline{\mathcal{M}}_{g,A[n]})$ there exists an unique $\tilde{\phi}\in\Aut(\overline{M}_{g,A[n]})$ such that $\pi\circ\phi = \tilde{\phi}\circ\pi$. So we get a morphism of groups
$$\tilde{\chi}:\Aut(\overline{\mathcal{M}}_{g,A[n]})\rightarrow\Aut(\overline{M}_{g,A[n]}).$$
\begin{Proposition}\label{inj}
If $2g-2+n\geq 3$ then the morphism $\tilde{\chi}$ is injective.
\end{Proposition}
\begin{proof}
For the values of $g$ and $n$ we are considering $\overline{\mathcal{M}}_{g,A[n]}$ is a normal Deligne-Mumford stack with trivial generic stabilizer. To conclude it is enough to apply \cite[Proposition A.1]{FMN}.
\end{proof}
By Proposition \ref{inj} for any $g\geq 1, n$ such that $2g-2+n\geq 3$ the group $\Aut(\overline{\mathcal{M}}_{g,A[n]})$ is a subgroup of $\mathcal{S}_{A[n]}$. Note that an admissible transposition $i\leftrightarrow j$ defines an automorphism of $\overline{\mathcal{M}}_{g,A[n]}$. Indeed, the contraction of a rational tail with three special points, that is a nodally attached rational tail with two marked points, does not affect either the coarse moduli space or the stack because it is a bijection on points and preserves the automorphism groups of the objects. However, it may induce a non trivial transformation on the universal curve.
\begin{Theorem}\label{autstack}
The automorphism group of the stack $\overline{\mathcal{M}}_{g,A[n]}$ is isomorphic to the group of admissible permutations
$$\Aut(\overline{\mathcal{M}}_{g,A[n]})\cong \mathcal{S}_{A[n]}$$
for any $g\geq 1$ and $n$ such that $2g-2+n\geq 3$. Furthermore, $\Aut(\overline{\mathcal{M}}_{1,A[1]})\cong \mathbb{C}^{*}$ while $\Aut(\overline{\mathcal{M}}_{1,A[2]})$ is trivial.
\end{Theorem}
\begin{proof}
By Proposition \ref{inj} the surjective morphism
$$\tilde{\chi}:\Aut(\overline{\mathcal{M}}_{g,A[n]})\rightarrow \mathcal{S}_{A[n]}$$
is an isomorphism. The isomorphism $\Aut(\overline{\mathcal{M}}_{1,A[1]})\cong \mathbb{C}^{*}$ derives from $\overline{\mathcal{M}}_{1,A[1]}\cong\overline{\mathcal{M}}_{1,1}\cong\mathbb{P}(4,6)$. Since a rational tail with three special points, that is a nodally attached rational tail with two marked points, is automorphisms-free the reduction morphism
$$\rho:\overline{\mathcal{M}}_{1,2}\rightarrow\overline{\mathcal{M}}_{1,A[2]}$$
is a bijection on points and preserves the automorphism groups of the objects. The stacks $\overline{\mathcal{M}}_{1,2}$ and $\overline{\mathcal{M}}_{1,A[2]}$ are isomorphic. We conclude by \cite[Proposition 4.5]{Ma}.
\end{proof}


\begin{thebibliography}{99999}
\bibitem[AM]{AM} \bibaut{C. Araujo, A. Massarenti}, \textit{Explicit log Fano structures on blow-ups of projective spaces}, \arXiv{1505.02460v1}.
\bibitem[BM1]{BM1} \bibaut{A. Bruno, M. Mella}, \textit{On some fibrations of $\overline{M}_{0,n}$}, \arXiv{1105.3293v1}.
\bibitem[BM2]{BM2} \bibaut{A. Bruno, M. Mella}, \textit{The automorphism group of $\overline{M}_{0,n}$}, J. Eur. Math. Soc. Volume 15, Issue 3, 2013, 949-968.
\bibitem[Co]{Co} \bibaut{D. A. Cox}, \textit{The homogeneous coordinate ring of a toric variety}, J. Algebraic Geom. 4, 1995, no. 1, 17-50.
\bibitem[Do]{Do} \bibaut{I. V. Dolgachev}, \textit{Classical Algebraic Geometry: A Modern View}, Cambridge University Press, 2012.
\bibitem[DI]{DI} \bibaut{I. V. Dolgachev, V. A. Iskovskikh}, \textit{Finite subgroups of the plane Cremona group}, Algebra, arithmetic, and geometry: in honor of Yu. I. Manin. Vol. I, 443-548, Progr. Math., 269, Birkh\"auser Boston, Inc., Boston, MA, 2009.
\bibitem[FMN]{FMN} \bibaut{B. Fantechi, E. Mann, F. Nironi}, \textit{Smooth toric DM stacks}, J. Reine Angew. Math. 648, 2010, 201-244.
\bibitem[GKM]{GKM} \bibaut{A. Gibney, S. Keel, I. Morrison}, \textit{Towards the ample cone of $\overline{M}_{g,n}$}, J. Amer. Math. Soc. 15, 2002, 273-294.
\bibitem[GP]{GP} \bibaut{G. Gonzalez-Sprinberg, I. Pan}, \textit{On the monomial birational maps of the projective space}, An. Acad. Brasil. Cienc. 75, 2003, no. 2, 129-134.
\bibitem[Ha]{Ha} \bibaut{B. Hassett}, \textit{Moduli spaces of weighted pointed stable curves}, Advances in Mathematics, 173, 2003, Issue 2, 316-352.
\bibitem[Ka]{Ka} \bibaut{M. Kapranov}, \textit{Veronese curves and Grothendieck-Knudsen moduli spaces $\overline{M}_{0,n}$}, Jour. Alg. Geom. 2, 1993, 239-262.
\bibitem[Kn]{Kn} \bibaut{F. Knudsen}, \textit{The projectivity of the moduli space of stable curves II: the stack $M_{g,n}$}, Math. Scand. 52, 1983, 161-199.
\bibitem[Li1]{Li1} \bibaut{V. Lin}, \textit{Algebraic functions, configuration spaces, Teichm\"uller spaces, and new holomorphically combinatorial invariants}, Funct. Anal. Appl, 45, 2011, 204-224.
\bibitem[Li2]{Li2} \bibaut{V. Lin}, \textit{Configuration spaces of $\mathbb{C}$ and $\mathbb{CP}^1$: some analytic properties}, \arXiv{math/0403120v3}.
\bibitem[LM]{LM} \bibaut{A. Losev, Y. Manin}, \textit{New moduli spaces of pointed curves and pencils of flat connections}, Michigan Math. J. Volume 48, Issue 1, 2000, 443-472.
\bibitem[Ma]{Ma} \bibaut{A. Massarenti}, \textit{The automorphism group of $\overline{M}_{g,n}$}, J. London Math. Soc. 2, 89, 2014, 131-150.
\bibitem[MM]{MM} \bibaut{A. Massarenti, M. Mella}, \textit{On the automorphisms of moduli spaces of curves}, Automorphisms in Birational and Affine Geometry, Springer Proceedings in Mathematics \& Statistics, 79, 2014, 149-167.
\bibitem[Mok]{Mok}\bibaut{S. Mochizuki}, \textit{Correspondences on hyperbolic curves}, J. Pure Applied Algebra, 131, 1998, 227-244.
\bibitem[Moo]{Moo}\bibaut{H. Moon}, \textit{A family of divisors on $\overline{M}_{g,n}$ and their log canonical models}, J. Pure Appl. Algebra, 219, 2015, no. 10, 4642-4652.
\bibitem[Ro]{Ro} \bibaut{H.L. Royden}, \textit{Automorphisms and isometries of Teichm\"uller spaces}, Advances in the theory of Riemann surfaces Ed. by L. V. Ahlfors, L. Bers, H. M. Farkas, R. C. Gunning, I. Kra, H. E. Rauch, Annals of Math. Studies, no.66, 1971, 369-383.
\end{thebibliography}
\end{document}